\documentclass[12pt,reqno]{amsart}
\usepackage{soul}
\usepackage{multicol,amsmath,graphics, mathtools, cancel,soul,color,bbm,amsfonts,dsfont,tikz,mathrsfs,amssymb,bm,hyperref,comment}

\usepackage[left=1in, right=1in, top=1.1in,bottom=1.1in]{geometry}
\usetikzlibrary{matrix,patterns,positioning}
\numberwithin{equation}{section}

\hypersetup{
    colorlinks=true, 
    linktoc=all,     
    linkcolor=blue,  
}

\newcommand{\E}{\mathbb{E}}

\newcommand{\N}{\mathbb{N}}
\newcommand{\Pb}{\mathbb{P}}

\newcommand{\R}{\mathbb{R}}

\newcommand{\vertiii}[1]{{\left\vert\kern-0.25ex\left\vert\kern-0.25ex\left\vert #1
    \right\vert\kern-0.25ex\right\vert\kern-0.25ex\right\vert}}

\def\R{\mathbb{R}}

\def\PP{\mathbb{P}}

\def\dh2l{\mathbf{d}_{\mathbb{H}_{2\ell}}}
\def\d2{\mathbf{d}_2}

\newtheorem{thm}{Theorem}[section]
\newtheorem{lemma}[thm]{Lemma}
\newtheorem{cor}[thm]{Corollary}
\newtheorem{prop}[thm]{Proposition}

\newtheorem{theorem}{Theorem}[section]

\theoremstyle{remark}
\newtheorem{remark}[theorem]{Remark}
\theoremstyle{definition}

\newcounter{dummy} \numberwithin{dummy}{section}

\newtheorem{Corollary}[dummy]{Corollary}
\newtheorem{Lemma}[dummy]{Lemma}

\def\1{\mathbbm{1}}

\makeatletter
\@namedef{subjclassname@1991}{2020 Mathematics Subject Classification}
\makeatother

\begin{document}

\title{Attenuated Poisson Dirichlet approximations for divisibility configurations}

\author{Victor Bernal, David Flores, Arturo Jaramillo}
\address{Arturo Jaramillo: Department of Probability and statistics, Centro de Investigaci\'on en matem\'aticas (CIMAT)}
\email{jagil@cimat.mx}

\address{David Torres Flores: Escuela Nacional de Estudios Superiores, Unidad Le\'on, Guanajuato}
\email{dtorres@cimat.mx}

\address{Victor Bernal: Department of mathematics, UC Santa Barbara}
\email{victorbernal749@gmail.com}

\date{\today}


\begin{abstract}
We study the point process formed by the normalized logarithms of the distinct prime factors of a harmonic random sample. We prove a quantitative convergence result, in a Wasserstein-type metric over decreasing sequences, toward the atom sequence of a Dickman Poisson cloud conditioned to have total mass at most one, equivalently a uniformly attenuated Poisson-Dirichlet law. The proof is based on the conditioned geometric representation of harmonic samples, a Poisson approximation chain for the associated point processes, monotone couplings of Poisson point processes, and Kolmogorov estimates for the Dickman approximation of weighted geometric sums.
\end{abstract}

\maketitle\section{Introduction}

The purpose of this paper is to contribute to the understanding of the statistical behavior of the prime factors arising in random harmonic samples. 
Our approach is based on representing the prime factors of random harmonic samples through an associated Poisson point process. This perspective originates from classical models with a similar structure, such as the ranked cycle lengths of a random permutation normalized by $n$; see \cite{kingman1975random, arratia2003}. A natural analogue in number theory arises from the ranked prime factors of a uniformly chosen integer, normalized by $\log n$; see \cite{billingsley1973} and \cite{arratia2003}. In both cases, one obtains a random sequence of normalized masses that can be encoded as a point process which converges asymptotically to a Poisson-Dirichlet distribution.\\

\noindent Our objective is to study this second model in the setting of harmonic samples and to establish a quantitative rate of convergence, in a suitable Wasserstein-type distance, toward a Poisson-Dirichlet distribution conditioned on a natural mass constraint. The classical uniform integer model admits quantitative metric estimates, as discussed below. Our goal is to develop the corresponding quantitative theory in the harmonic setting, where the limiting object is modified by a natural random mass constraint.\\

\noindent To describe in more detail the results from \cite{kingman1975random,arratia2003,billingsley1973}, we recall the stick-breaking construction associated with the Poisson-Dirichlet law. Let $\{U_k\ ;\ k\geq 1\}$ be independent and identically distributed random variables with uniform distribution on $(0,1)$. Define
\begin{align}\label{eq:GEMstickbreakingdef}
L_k
=
U_k\prod_{i<k}(1-U_i),
\qquad k\geq 1,
\end{align}
where the empty product is interpreted as one. Let $(V_k\ ;\ k\geq 1)$ denote the decreasing rearrangement of $(L_k\ ;\ k\geq 1)$. The sequence $(V_k\ ;\ k\geq 1)$ has the classical Poisson-Dirichlet distribution with parameter one. We encode this mass partition as the random measure
\begin{align}\label{eq:mathcalXdirichdef}
\mathcal{X}
  &=\sum_{k\geq 1}\delta_{V_k}.
\end{align}
For convenience, we adopt this construction as a working representation and refer the reader to \cite{sethuraman1994} for further discussion. We now recall the results mentioned above in a form that will be useful for our purposes. In the sequel, we denote by \(\mathcal{M}((0,\infty))\) the space of Radon measures on \((0,\infty)\), endowed with the vague topology.

\subsection{Cycle factorizations of random permutations}

\noindent As a starting point, we recall the classical result describing the cycle structure of random permutations, stated in the point-process formulation that will serve as a reference for the presentation of our main results.

\begin{theorem}\label{startArratia2perm}
Let $\pi_n$ be a uniformly random permutation of $\{1,\dots,n\}$ and let
\begin{align*}
L_1^{(n)} \geq L_2^{(n)} \geq \cdots \geq L_{K_n}^{(n)}
\end{align*}
denote the cycle lengths of $\pi_n$ arranged in decreasing order, where
$K_n$ is the number of cycles of $\pi_n$. Define the random measures
\begin{align*}
\mathcal{X}_n
:=
\sum_{k=1}^{K_n}\delta_{L_k^{(n)}/n}.
\end{align*}
Then $\mathcal{X}_n$ converges in law to a Poisson-Dirichlet distribution.
\end{theorem}
\noindent On the space of point measures on $(0,\infty)$ assigning finite mass to compact sets, which is the natural state space of $\mathcal{X}_n$, one may introduce several natural notions of distance. For the purposes of this paper, we consider the metric described in Section~\ref{se:prelimpointprocc}, obtained by viewing the atoms of a measure as a decreasing sequence and equipping this sequence with an $L^1$-type metric. This structure naturally induces the corresponding notion of $1$-Wasserstein distance.\\

\noindent Having recalled the qualitative convergence described in Theorem~\ref{startArratia2perm}, we turn to the problem of quantifying this convergence under a suitable metric. A fundamental result in this direction, established in \cite{arratia1997b}, shows that, with respect to a suitable metric denoted by $\mathrm{dist}$,
\begin{align*}
\mathrm{dist}(\mathcal{X}_n, \mathcal{X})
\;\le\;
C\,\frac{\log n}{n},
\end{align*}
for some constant $C>0$, where $\mathcal{X}$ denotes the Poisson-Dirichlet limit. The metric under consideration will turn out to be precisely the one introduced in Section~\ref{se:prelimpointprocc}. We now present the corresponding result in the number-theoretic setting.

\subsection{Prime factorizations of integer uniform samples}

\noindent Observe that the main structural features of the discussion above are the following: we considered objects that admit a canonical decomposition into elementary components, namely cycles; we arranged these components according to their sizes; and we observed that the resulting collection exhibits a universal limiting behavior when viewed as a point process. It is therefore natural to investigate whether similar limiting structures arise in other settings where objects admit multiplicative decompositions. For the purposes of the present paper, the most relevant instance is the following. The result can be consulted in \cite{billingsley1973} and \cite{arratia2003}.
	
\begin{theorem}\label{startArratia2}
Let $J_n$ be uniformly distributed on $\{1,\dots,n\}$ and let
\begin{align*}
P_1(J_n) \ge P_2(J_n) \ge \cdots \ge P_{K_n}(J_n)
\end{align*}
denote the prime factors of $J_n$, counted with multiplicity and arranged in decreasing order, where $K_n$ is the number of prime factors of $J_n$ counted with multiplicity. Define the random measure
\begin{align*}
\mathcal{X}_n
:=
\sum_{k=1}^{K_n}\delta_{\log P_k(J_n)/\log n}.
\end{align*}
Then $\mathcal{X}_n$ converges in law to a Poisson-Dirichlet distribution.
\end{theorem}

\noindent The preceding theorem is stated in the classical formulation, where prime factors are counted with multiplicity. In the present paper, however, we work with the set of distinct prime divisors. This is the natural formulation for the point-process representation used below. The distinction between prime factors counted with multiplicity and distinct prime divisors does not affect the limiting behavior on the logarithmic scale, since repeated prime powers give a negligible contribution after normalization by $\log n$.\\

\noindent The arguments underlying Theorem~\ref{startArratia2} are closely related to the classical Kubilius method; see \cite{kubilius1964, tenenbaum1999, chen2023kubilius}. This method approximates the joint law of finitely many prime multiplicities by independent geometric random variables with rapidly decaying total variation error. Such a reduction replaces intricate divisibility questions with tractable problems involving independent variables. However, in the form needed here, the method gives finite-dimensional control: it controls fixed tuples of primes, typically in the regime of small prime factors.\\

\noindent Quantitative estimates are also known for the Poisson-Dirichlet convergence theorem for uniform random integers, originally due to Billingsley. In that classical formulation, the prime factors are listed with multiplicity. More precisely, if \(J_n\) is uniformly distributed on \(\{1,\dots,n\}\), one writes
\begin{align*}
J_n
=
P_1(J_n)P_2(J_n)\cdots,
\end{align*}
where the factors are arranged in non-increasing order, and the sequence is completed by setting the remaining entries equal to one. Thus repeated prime powers are encoded by repeated entries in the sequence. This point will be important below, since the harmonic model studied in the present paper keeps only distinct prime divisors.\\

\noindent For the classical uniform model, Arratia~\cite{MR1919568} constructed a coupling between the ordered prime factors of \(J_n\) and a Poisson-Dirichlet mass partition \((V_i\ ;\ i\geq 1)\) such that, for some constant \(C>0\),
\begin{align*}
\mathbb E \sum_{i\ge1}
\left|
\log P_i(J_n)-(\log n)V_i
\right|
\leq
C\log\log n.
\end{align*}
Equivalently, after normalization by \(\log n\), this yields a bound of order $\frac{\log\log n}{\log n}$ in the corresponding \(\ell^1\)-type metric for the normalized logarithmic prime factors. More recently, Haddad and Koukoulopoulos~\cite{MR4963450} sharpened Arratia's coupling to the optimal order. More precisely, they showed that there exists a constant \(C>0\) and a coupling such that, for all sufficiently large \(n\),
\begin{align*}
\mathbb E \sum_{i\ge1}
\left|
\log P_i(J_n)-(\log n)V_i
\right|
\leq
C.
\end{align*}
They also proved that this order is optimal, in the sense that the infimum of the corresponding expected coupling cost over all admissible couplings is bounded away from zero as \(n\to\infty\).\\

\noindent These results provide the closest reference point for the present work, but they do not directly address the harmonic model considered here, where the total logarithmic mass remains random in the limit and where the point process is built from distinct prime divisors.\\

\noindent A natural approach is to work in a probabilistic framework where the multiplicative structure admits a representation in terms of independent variables, possibly conditioned on a suitable event; see \cite{chen2021, bernal2025, jaramilloyang2024}. We now formalize this perspective by introducing the harmonic model and stating our main results.

\section{Main results}\label{sec:mainresults}

\noindent In this work, we consider harmonic samples instead of the uniform distribution on $[n]:=\{1,\dots,n\}$. As discussed in Section \ref{se:repofharmonicprimefact}, this model retains the relevant statistical features of random integers while providing a tractable computational framework. Moreover, results in the harmonic setting can be transferred to the uniform model through a total variation approximation obtained by multiplying the sample by a suitable prime factor, as explained in full detail in \cite{chen2021}. We now define the harmonic model. Let \(H_n\) be a random variable defined over a given probability space $(\Omega,\mathcal{F},\mathbb{P})$, with 
\begin{align*}
\mathbb{P}[H_n = k] = \frac{1}{\mathfrak{l}_n k},
\end{align*}
for $k = 1,\dots,n$, where
\begin{align*}
\mathfrak{l}_n
:=
\sum_{j=1}^n \frac{1}{j}
\end{align*}
is the normalizing constant. Each realization of \(H_n\) admits the unique prime factorization
\begin{align*}
H_n = \prod_{p \le n} p^{\alpha_p(H_n)},
\end{align*}
for some \(\alpha_p(H_n)\) belonging to $\mathbb{N}_0$. Denote by $\mathcal{P}$ the set of prime numbers and let $\mathfrak{P}(h)$ denote the element in $2^{\mathcal{P}}$, defined through 
\begin{align*}
\mathfrak{P}(h)
:=
\{ p\in\mathcal{P}  \ ;\  p\  \textit{divides } h\},
\end{align*}
for $h$ in the natural numbers. Our specific goal is to describe the typical shape of the set $\mathfrak{P}(H_n)$. Given a function \(f:(0,\infty)\to\mathbb{R}\) and a subset \(U\subset(0,\infty)\), we denote by
\begin{align*}
f(U)
:=
\{f(x)\ ;\ x\in U\}
\end{align*}
the image of \(U\) under \(f\). We then consider the renormalized set
\begin{align}\label{setnormdef}
\log(\mathfrak{P}(H_n))/\log(n),	
\end{align}
and show that it exhibits a non-trivial random asymptotic behavior in distribution as $n$ tends to infinity. In order to shed clarity on the precise manner in which a random set may admit a meaningful limit, without resorting to intricate metric structures over families of sets, we adopt a representation based on embeddings into spaces of measures. Observe that a locally finite subset $U$ of $(0,\infty)$ can be naturally identified with the point measure in $\mathcal{M}((0,\infty))$, through
\begin{align*}
\iota[U] := \sum_{x\in U} \delta_x.
\end{align*}
The advantage of transferring the problem from describing sets to describing elements in $\mathcal{M}((0,\infty))$ is that it allows one to study scaling limits through normalizations induced by push-forward transformations. More precisely, given a sequence of random measures 
\begin{align*}
\{\iota\circ \mathfrak{P}(H_n) \ ;\ n\ge 1\},
\end{align*}
one may seek normalizing maps $T^n$ for which the push-forward measures
\begin{align*}
T^n_{\#}\big[\iota\circ \mathfrak{P}(H_n)\big]
\end{align*}
have a non-trivial asymptotic behavior. The first step is to identify the relevant scale. The uniform model recalled above suggests the logarithmic scale: there, prime factors are normalized by $\log n$, leading to the Poisson-Dirichlet limit. It is therefore natural to ask whether the same scale also yields a meaningful limit in the harmonic setting.\\

\noindent Before adopting this normalization, however, it is important to point out a fundamental difference between the uniform and harmonic models. In the uniform case, the logarithmic normalization leads to a limiting mass partition whose total mass is equal to one. In the harmonic case, the same normalization leads to a different behavior. Indeed, at the level of the total logarithmic scale, the definition of the harmonic weights gives
\begin{align*}
\Pb\left[\log H_n/\log n\le t\right]
  &=
\frac{\mathfrak{l}_{\lfloor n^t\rfloor}}{\mathfrak{l}_n},
\end{align*}
for $0<t<1$. Since $\mathfrak{l}_m$ grows logarithmically in $m$, the right-hand side converges to $t$. Thus the total logarithmic scale of $H_n$ does not concentrate at one, but has a non-degenerate limiting behavior. Consequently, the total logarithmic mass available to the prime factors should not be expected to be asymptotically fixed. The limiting object should therefore allow the total logarithmic mass to be random. Thus, the natural candidate is not an ordinary Poisson-Dirichlet mass partition, but an attenuated Poisson-Dirichlet type object.\\

\noindent With this distinction in mind, we adopt the logarithmic normalization
\begin{align*}
T^n(x):=\frac{\log x}{\log n}.
\end{align*}
In this case, a direct computation shows that
\begin{align*}
T^n_{\#}\big[\iota\circ \mathfrak{P}(H_n)\big]
=
\iota\!\left[\frac{\log(\mathfrak{P}(H_n))}{\log n}\right].
\end{align*}
Consequently, studying the scaling limit of the left-hand side is equivalent to analyzing \eqref{setnormdef}.\\

\noindent A particularly transparent formulation of our main results is obtained by working with ordered sequences. To make this precise, we introduce the space 
\begin{align*}
\mathcal{S}
=
\left\{
(x_i\ ;\ i\ge1)\in[0,\infty)^{\mathbb{N}}
\ ;\
x_i\ge x_{i+1},
\quad
\sum_{i\ge1}x_i\le 1
\right\}.
\end{align*}
Since the elements of \(\mathfrak{P}(H_n)\) are distinct, they can be uniquely arranged in decreasing order. We denote by
\begin{align*}
\boldsymbol{\Xi}^{n}
=
(\Xi_1^n,\Xi_2^n,\dots)
\end{align*}
the sequence whose coordinates list the elements of \(\mathfrak{P}(H_n)\) in decreasing order, with the convention that $\Xi_k^n=1$ when fewer than $k$ prime divisors are present. This representation contains exactly the same information as the random set \(\mathfrak{P}(H_n)\). The normalized sequence $\log\boldsymbol{\Xi}^n/\log n$ is $\mathcal S$-valued. Let \(\vartheta_n\) denote its law. Over \(\mathcal{S}\), we consider the metric
\begin{align*}
d_{\mathcal{S}}(\mathbf{x},\mathbf{y})
=
\sum_{k\ge1}|x_k-y_k|.
\end{align*}
This metric induces the corresponding $1$-Wasserstein distance $d_{W,\mathcal{S}}$ on the space $\mathcal{M}_1(\mathcal{S})$ of probability measures over $\mathcal{S}$, defined by
\begin{align*}
d_{W,\mathcal{S}}(\mu,\nu)
:=
\inf_{\pi \in \Pi(\mu,\nu)}
\int_{\mathcal{S}\times\mathcal{S}}
d_{\mathcal{S}}(\mathbf{x},\mathbf{y})\,
\pi(d\mathbf{x},d\mathbf{y}),
\end{align*}
where $\Pi(\mu,\nu)$ denotes the set of probability measures on $\mathcal{S}\times\mathcal{S}$ with marginals $\mu$ and $\nu$. Let \(\theta_{\mathcal{S}}\) denote the distribution on \(\mathcal{S}\) of a Poisson-Dirichlet mass partition arranged in decreasing order.\\

\noindent Furthermore, we introduce an attenuation operation on laws over $\mathcal S$. Let $\gamma$ be an element of $\mathcal M_1(\mathcal S)$, and let $\mathbf X=(X_1,X_2,\dots)$ be an $\mathcal S$-valued random variable with distribution $\gamma$. Let $U$ be a uniform random variable on $(0,1)$, independent of $\mathbf X$. We define the uniform attenuation of $\gamma$ by
\begin{align*}
\mathfrak A[\gamma]
:=
\mathcal L(U\mathbf X),
\end{align*}
where
\begin{align*}
U\mathbf X:=(UX_1,UX_2,\dots).
\end{align*}
This operation preserves the order of the atoms but allows the total mass to be randomly attenuated. In particular, when $\gamma=\theta_{\mathcal S}$, we call $\mathfrak A[\theta_{\mathcal S}]$ the uniformly attenuated Poisson-Dirichlet law.\\ 

\noindent With this discussion in mind, we can now present our main result, which reads as follows.

\begin{theorem}\label{theorem:main}
There exist constants \(C>0\) and $N\in\N$ such that, for all $n\geq N$, 
\begin{align*}
d_{W,\mathcal{S}}\big(\vartheta_n,\mathfrak{A}[\theta_{\mathcal S}]\big)
\le
\frac{C}{\log n}\,(\log\log n)^{\mathrm e-\frac 32}.
\end{align*}
\end{theorem}

\begin{remark}
The rate in Theorem \ref{theorem:main} mirrors the classical permutation setting. Arratia \cite{arratia1997b} proved that the $1$-Wasserstein distance between normalized cycle lengths and the Poisson-Dirichlet limit is of order $\log n/n$, where the normalization scale is $n$. In the present setting the normalization scale is $\log n$, yielding the rate
\begin{align*}
\frac{1}{\log n}(\log\log n)^{\mathrm e-\frac32},
\end{align*}
which has the same general form: a logarithmic correction of the normalization scale divided by the scale itself.
\end{remark}

\begin{remark}
As a consequence of Theorem~\ref{theorem:main}, one obtains quantitative convergence for all Lipschitz observables of the normalized prime factors. Indeed, by the Kantorovich-Rubinstein duality, if $\mathbf{X}$ is an $\mathcal{S}$-valued random variable with distribution $\mathfrak A[\theta_{\mathcal S}]$, then for every Lipschitz function $F:\mathcal{S}\to\mathbb{R}$,
\begin{align*}
\big|
\mathbb{E}\big[F\big(\log(\boldsymbol{\Xi}^n)/\log n\big)\big]
-
\mathbb{E}[F(\mathbf{X})]
\big|
\le
C\|F\|_{\mathrm{Lip}}\,
\frac{1}{\log n}(\log\log n)^{\mathrm e-\frac32},
\end{align*}
where $\|F\|_{\mathrm{Lip}}$ denotes the Lipschitz constant of $F$ with respect to $d_{\mathcal S}$. 
The coordinate projection $F(\mathbf{x})=x_k$ is $1$-Lipschitz on $\mathcal{S}$. Therefore the $k$-th coordinate of the normalized sequence,
\begin{align*}
\frac{\log \Xi_k^n}{\log n},
\end{align*}
converges quantitatively toward the $k$-th coordinate of the limit distribution.
\end{remark}

\noindent Following the notation from Section \ref{se:prelimpointprocc}, one may also define a metric $\mathfrak{d}$ on the set of point measures in $\mathcal{M}((0,\infty))$, given by \eqref{eq:dfrakdef}, which in turn induces a natural $1$-Wasserstein distance, denoted by $d_W$. We use the same notation \(\mathfrak A\) for the corresponding attenuation operation on point-measure laws: if \(\mathcal X=\sum_{k\ge1}\delta_{V_k}\), then \(\mathfrak A[\mathrm{Law}(\mathcal X)]\) denotes the law of
\begin{align*}
\sum_{k\ge1}\delta_{U V_k},
\end{align*}
where \(U\) is uniform on \((0,1)\) and independent of \(\mathcal X\). As a consequence of the discussion at the end of Section \ref{se:prelimpointprocc},
\begin{align*}
d_W\!\left(
\mathrm{Law}\big(T^n_{\#}[\iota\circ \mathfrak{P}(H_n)]\big),
\mathfrak{A}[\mathrm{Law}(\mathcal{X})]
\right)
=
d_{W,\mathcal{S}}\big(\vartheta_n,\mathfrak{A}[\theta_{\mathcal S}]\big).
\end{align*}

\noindent In particular, the inequality stated in Theorem \ref{theorem:main} yields
\begin{align*}
d_W\left(
\mathrm{Law}\big(T^n_{\#}[\iota\circ \mathfrak{P}(H_n)]\big),
\mathfrak{A}[\mathrm{Law}(\mathcal{X})]
\right)
\le
\frac{C}{\log n}\,(\log\log n)^{\mathrm e-\frac 32}.
\end{align*}

\noindent The rest of the paper is organized as follows. In Section \ref{se:prelimpointprocc}, we collect the point-process notation and introduce the Wasserstein-type metrics used throughout the paper. In Section \ref{se:repofharmonicprimefact}, we recall the representation of harmonic prime factors in terms of independent geometric random variables conditioned on a global constraint, and we also discuss the Dickman approximation estimates, Mertens-type bounds, and the Poisson-Dirichlet background needed for the proof. Finally, in Section \ref{sec:8}, we prove Theorem \ref{theorem:main} by constructing a chain of point-process approximations and controlling both the Wasserstein error and the discrepancy of the conditioning event.

\section{Preliminaries}

\subsection{Point processes}\label{se:prelimpointprocc}

The material in this section is classical in the theory of stochastic processes; see, for instance, \cite{kyprianou2014} for further details. Let $\mathbbm{X}:=(0,1]$, and let $\mathbf{N}(\mathbbm{X})$ denote the space of counting measures on $\mathbbm{X}$ that are finite on compact subsets. We equip this space with its canonical $\sigma$-field. Let $\mathcal{M}(\mathbbm{X})$ denote the space of Radon measures on $\mathbbm{X}$. A measure $\mu \in \mathcal{M}(\mathbbm{X})$ is called a L\'evy measure if it is finite on intervals bounded away from zero and satisfies
\begin{align*}
\int_{\mathbbm{X}} (1\wedge y^2)\,\mu(dy) < \infty.
\end{align*}
We denote by $\mathrm{Lv}(\mathbbm{X})$ the collection of all such measures.  It is well known that, for every $\mu $ in $ \mathrm{Lv}(\mathbbm{X})$, there exists a probability measure $\rho_{\mu}$ on $\mathbf{N}(\mathbbm{X})$ corresponding to a Poisson point process with intensity measure $\mu$. We say that two L\'evy measures $\mu,\nu$ over $\mathbbm{X}$ satisfy $\mu\preceq\nu$ if
\begin{align*}
\mu[(t,1]] \le \nu[(t,1]],
\qquad t\in(0,1].
\end{align*}
Every $\eta$ in $\mathbf{N}(\mathbbm{X})$ admits a representation
\begin{align*}
\eta = \sum_{i \in I} k_i \,\delta_{x_i},
\end{align*}
where $I$ is at most countable, $\{x_i : i \in I\}$ is the support of $\eta$,
and $k_i$ in $\mathbb{N}$ denote the corresponding multiplicities. The coefficients $k_i$ are integer-valued since $\eta$ is a counting measure, and they may be strictly larger than one whenever the control measure has atoms.\\

\subsubsection{Wasserstein-type metrics over point measures}

\noindent Given $\eta$ in $\mathbf{N}(\mathbbm{X})$, define its decreasing multiplicity list $\mathfrak{q}(\eta)=(\mathfrak{q}_n(\eta)\ ;\ n\ge1)$ in $[0,1]^{\mathbb{N}}$ as follows.
Write
\begin{align*}
\eta=\sum_{i\in I} k_i\,\delta_{x_i},
\end{align*}
for $k_i$ in $\mathbb{N}$, with $\{x_i:i\in I\}=\mathrm{supp}(\eta)$. Form the multiset in which each $x_i$ appears with multiplicity $k_i$, and list its elements in non-increasing order. If $\eta[\mathbbm{X}]$ is finite, fill the list with zeros so as to obtain an infinite sequence. Define, for $\eta,\zeta$ in $\mathbf{N}(\mathbbm{X})$, the extended metric
\begin{align}\label{eq:dfrakdef}
\mathfrak{d}(\eta,\zeta)
:=
\inf_{\sigma\in\mathfrak{S}(\mathbb{N})}
\sum_{n=1}^{\infty}\big|\mathfrak{q}_n(\eta)-\mathfrak{q}_{\sigma(n)}(\zeta)\big|,
\end{align}
where $\mathfrak{S}(\mathbb{N})$ denotes the set of all permutations of $\mathbb{N}$. If all multiplicities are equal to one, the identity permutation is optimal and
\begin{align*}
\mathfrak{d}(\eta,\zeta)
=
\sum_{n=1}^{\infty}
|x_n(\eta)-x_n(\zeta)|,
\end{align*}
where $x_n(\eta)$ and $x_n(\zeta)$ denote the atoms of $\eta$ and $\zeta$, respectively, arranged in non-increasing order and completed with zeros when necessary. This follows from the rearrangement inequality. If $\mathrm{supp}(\eta)=\mathrm{supp}(\zeta)$, then
\begin{align}\label{eq:acouplingbounddfrak}
\mathfrak{d}(\eta,\zeta)
\le
\sum_{i\in I} |k_i-\ell_i|\, x_i,
\end{align}
where $k_i,\ell_i$ denote the multiplicities of $x_i$ in $\eta$ and $\zeta$, respectively. The distance $\mathfrak{d}$ naturally induces a Wasserstein distance in $\mathbf{N}(\mathbbm{X})$, denoted by $d_{W}$. More explicitly, if $\alpha$ and $\beta$ are probability measures on $\mathbf{N}(\mathbbm{X})$, we set
\begin{align*}
d_W(\alpha,\beta)
:=
\inf_{\pi\in\Pi(\alpha,\beta)}
\int_{\mathbf{N}(\mathbbm{X})\times\mathbf{N}(\mathbbm{X})}
\mathfrak d(\eta,\zeta)\,\pi(d\eta,d\zeta).
\end{align*}
We have that for any L\'evy measures $\mu,\nu$,
\begin{align}\label{eq:W_d_bound_by_xTV}
d_W(\rho_\mu,\rho_\nu)
\le
\int_{\mathbbm{X}} x\,|\mu-\nu|(dx).
\end{align}

\noindent This distance can be identified with the Wasserstein distance over probability measures on $\mathcal{S}$, as discussed in Section \ref{sec:mainresults}, whenever the induced ordered sequences belong to $\mathcal S$. More precisely, let $\mu$ be a L\'evy measure on $\mathbbm{X}$ that is absolutely continuous with respect to Lebesgue measure, and let $\rho_\mu$ denote the law of the associated Poisson point process on $\mathbf{N}(\mathbbm{X})$. In this case, the configurations are almost surely simple, so that the decreasing multiplicity list $\mathfrak{q}(\eta)$ coincides with the ordered sequence of atoms of $\eta$. If, in addition, this ordered sequence belongs to $\mathcal S$ almost surely, then the map
\begin{align*}
\eta \longmapsto \mathfrak{q}(\eta)
\end{align*}
gives a measurable representation of the point process as an $\mathcal{S}$-valued random element, and the metric $\mathfrak{d}$ defined in \eqref{eq:dfrakdef} agrees with $d_{\mathcal{S}}$ under this representation. Consequently, if $\mu,\nu$ are absolutely continuous L\'evy measures whose induced ordered sequences belong to $\mathcal S$ almost surely, and if $\vartheta_\mu$ and $\vartheta_\nu$ denote the laws of the corresponding $\mathcal{S}$-valued random elements, then
\begin{align}\label{eq:equivofwassersteins}
d_W(\rho_\mu,\rho_\nu)
=
d_{W,\mathcal{S}}(\vartheta_\mu,\vartheta_\nu).	
\end{align}

\subsubsection{Discrepancies over measures in $\mathbf{N}(\mathbbm{X})$}

\noindent In the sequel, given a pair of $\mathbf{N}(\mathbbm{X})$-valued random variables $(X,Y)$ defined on a common probability space and a measurable set $B\subset \mathbf{N}(\mathbbm{X})$, we define
\begin{align*}
\mathcal E_B[X,Y]
:=
\Pb \big[\{X\in B\}\,\Delta\,\{Y\in B\}\big].
\end{align*}
This quantity measures the discrepancy between $X$ and $Y$ over the set $B$ and satisfies the following triangle-type inequality.

\begin{Lemma}\label{lem:EB_chain_bound}
Let $B\subset \mathbf N(\mathbbm{X})$ be measurable, and let $X_0,\dots,X_m$ be $\mathbf N(\mathbbm{X})$-valued random variables defined on a common probability space. Then
\begin{align*}
\Pb\big[\{X_0\in B\}\Delta\{X_m\in B\}\big]
\le
\sum_{i=1}^m \Pb \big[\{X_{i-1}\in B\}\Delta\{X_i\in B\}\big].
\end{align*}
Equivalently,
\begin{align*}
\mathcal E_B[X_0,X_m]
\le
\sum_{i=1}^m \mathcal E_B[X_{i-1},X_i].
\end{align*}
\end{Lemma}

\noindent This ordering can be utilized to easily construct couplings of point processes, as indicated by the following result.

\begin{Lemma}\label{Propcouplmonotone}
Let $\mu,\nu$ be elements in $\mathrm{Lv}(\mathbbm{X})$, satisfying $\mu\preceq\nu$. Then there exists a coupling of two sequences $\{X_k^1;\,k\ge1\}$ and $\{X_k^2;\,k\ge1\}$ such that $X_k^1 \le X_k^2$ for all $k\ge1$, and
\begin{align*}
\rho_{\mu_i}
=
\mathrm{Law}\big(\sum_{k\ge1} \mathbbm{1}_{\{X_k^i>0\}}\delta_{X_k^i}\big),
\end{align*}
for $i=1,2$, where $\mu_1=\mu$ and $\mu_2=\nu$.
\end{Lemma}

\subsection{Representation of Harmonic prime factors}\label{se:repofharmonicprimefact}

We begin by recalling a probabilistic representation that will play a central role in our analysis. Let $H_n$ be a random variable with the harmonic distribution on $\{1,\dots,n\}$. A key feature of $H_n$ is that the vector of its valuations admits a convenient probabilistic description in terms of independent random variables subject to a single global constraint. Let $\{\xi_p: p\in\mathcal{P}\}$ be a family of independent geometric random variables satisfying
\begin{align*}
\PP[\xi_p=k]=(1-p^{-1})p^{-k}, \qquad k\in\mathbb{N}_0 .
\end{align*}
Define the event
\begin{align}\label{eq:Andef}
A_n
  :=
\big\{
\prod_{p\le n} p^{\xi_p}\le n
\big\}.
\end{align}
The following statement describes the divisibility structure of $H_n$. The corresponding proof can be consulted in \cite{chen2021}.

\begin{prop}\label{p:div_H}
Let $\vec{C}(n)=(\alpha_p(H_n):p\le n)$ and $\vec{\xi}(n)=(\xi_p:p\le n)$. Then for all $n\geq  21$
\begin{align}
\PP[A_n]
  &=
\mathfrak{l}_n\prod_{p\le n}(1-p^{-1})\ge \frac12 , \label{eq:PrAnineq}
\\
\mathcal{L}(\vec{C}(n))
  &=\mathcal{L}(\vec{\xi}(n)\mid A_n).
\label{eq:conditionallaws}
\end{align}
In particular, for any additive arithmetic function $\psi$,
\begin{align}\label{eq:conditionallaw20}
\mathcal{L}(\psi(H_n))
=
\mathcal{L}\big(
\sum_{p\le n}\psi(p^{\xi_p}) \,\big|\, A_n
\big).
\end{align}
\end{prop}

\noindent This representation reduces the analysis of harmonic prime factors to the study of independent geometric variables under the single conditioning event \(A_n\).

\subsection{Dickman approximations and weighted random sums}

The Dickman function $\rho$ plays a central role in probabilistic number theory, particularly in the study of smooth numbers and large prime factors. It is defined as the unique solution to the delay differential equation
\begin{align}\label{eq:dickman_def}
u\rho'(u)=-\rho(u-1),
\end{align}
for $u>1$, together with the initial condition $\rho(u)=1$ for $0\le u\le1$. This function is continuous on $\mathbb{R}_{+}$ and arises naturally in asymptotic problems involving multiplicative structures.\\

\noindent A probabilistic viewpoint is obtained by observing that $x\mapsto e^{-\gamma_{\mathrm E}}\rho(x)$ defines a probability density on $\mathbb{R}_{+}$, where $\gamma_{\mathrm E}$ denotes the Euler-Mascheroni constant. The associated probability law is known as the Dickman distribution. Further discussion of this interpretation can be found in \cite{MR3968515}; classical analytic properties are discussed for instance in \cite{MR3363366}.\\

\noindent Dickman-type limits also appear in approximations of weighted sums of independent random variables. Such approximations have been systematically studied using Stein's method in \cite{BattSchulte}, where explicit bounds in Kolmogorov distance are obtained. In the setting considered here, the random quantities that arise lead naturally to normalized sums of the form
\begin{align*}
S_m
=
\frac{1}{\log m}
\sum_{p\le m}\log(p)\xi_p.
\end{align*}
A quantitative analysis of these sums was obtained in \cite{jaramilloyang2024}. The following estimate will serve as one of the main probabilistic ingredients in our arguments.\\

\noindent For a locally integrable function $f:\mathbb{R}_+\to\mathbb{R}$ we denote by $\mathcal{I}[f]$ the integral operator
\begin{align*}
\mathcal{I}[f](x)
  :=\int_{0}^{x}f(s)\,ds .
\end{align*}

\begin{lemma}\label{lem:kolmmain}
There exists a universal constant $C>0$ such that, for all $z>0$,
\begin{align*}
|\Pb[S_{m}\leq z]-e^{-\gamma_{\mathrm E}}\mathcal{I}[\rho](z)|
\leq
\frac{C(1+1/z^2)}{\log m}.
\end{align*}
\end{lemma}

\subsection{Mertens' estimates}

We recall several classical estimates concerning partial sums and products over the prime numbers. These results originate in the work of Franz Mertens and are commonly known as the first, second, and third Mertens formulas. They provide asymptotic descriptions of sums involving the functions $\log(p)/p$ and $1/p$, as well as products of the form $\prod_{p\le m}(1-1/p)$. Let
\begin{align*}
\mathcal P_n
:=
\{p\in\mathcal P\ ;\ p\le n\}.
\end{align*}
As a consequence of the first Mertens formula, there exists a constant \(C>0\) such that, for all \(n\ge2\),
\begin{align}\label{Firstmertens}
\left|
\sum_{p\in\mathcal{P}_{n}}\frac{\log(p)}{p}
-\log(n)
\right|
  \le C.
\end{align}

The second formula concerns the harmonic sum over the primes. It asserts that
\begin{align}\label{Secondmertens}
\left|\sum_{p\in\mathcal{P}_n}\frac{1}{p}-\log\log(n)-c_1\right|
  \le \frac{5}{\log(n)},
\end{align}
where $c_1$ denotes an absolute constant. These estimates are standard in analytic number theory and will be used repeatedly in the sequel. Detailed proofs can be found in \cite{MR3363366}.

\subsection{Dirichlet processes and Poisson-Dirichlet distributions}

Poisson-Dirichlet distributions provide a natural framework for describing the ranked atom sizes of certain random discrete probability measures; see, for instance, \cite{kingman1975random,arratia2003}. We begin by recalling the Dirichlet distribution. Let $\alpha_1,\dots,\alpha_k>0$. A random vector $(X_1,\dots,X_k)$ taking values in the simplex
\begin{align*}
\big\{(x_1,\dots,x_k)\in[0,1]^k \ ;\ \sum_{i=1}^k x_i =1\big\}
\end{align*}
is said to have a Dirichlet distribution with parameters $(\alpha_1,\dots,\alpha_k)$ if its density with respect to Lebesgue measure on the simplex is proportional to
\begin{align*}
x_1^{\alpha_1-1}\cdots x_k^{\alpha_k-1}.
\end{align*}

Let $(E,\mathcal{E})$ be a measurable space, let $\pi$ be a probability measure on $E$, and let $\alpha>0$. A Dirichlet process with parameters $(\alpha,\pi)$ is a random probability measure $Z$ on $E$ satisfying the following property: for every finite measurable partition $(A_1,\dots,A_k)$ of $E$, the vector
\begin{align*}
\big(Z[A_1],\dots,Z[A_k]\big)
\end{align*}
has a Dirichlet distribution with parameters
\begin{align*}
(\alpha \pi[A_1],\dots,\alpha \pi[A_k]),
\end{align*}
with the usual degenerate interpretation when some $\pi[A_i]=0$. We write
\begin{align*}
Z \sim DP(\alpha,\pi).
\end{align*}
This construction was introduced by Ferguson \cite{ferguson1973}.\\

\noindent A fundamental representation of Dirichlet processes is obtained through random atomic measures. Let $(W_i)_{i\ge1}$ be independent random variables with
\begin{align*}
W_i\sim \mathrm{Beta}(1,\alpha),
\end{align*}
and define the stick-breaking weights
\begin{align*}
P_1=W_1,\qquad
P_i=W_i\prod_{j<i}(1-W_j),
\quad i\ge2.
\end{align*}
Let $(\xi_i)_{i\ge1}$ be an i.i.d.\ sequence of $E$-valued random variables with distribution $\pi$, independent of $(P_i)_{i\ge1}$. Then
\begin{align*}
Z := \sum_{i\ge1} P_i \delta_{\xi_i}
\end{align*}
is a Dirichlet process with parameters $(\alpha,\pi)$; see \cite{sethuraman1994,pitman2006}. The sequence $(P_i)$ is in size-biased order and has the GEM$(\alpha)$ distribution. Its decreasing rearrangement has the Poisson-Dirichlet distribution, denoted here by $PD(\alpha)$. Thus, in this representation, the random measure $Z$ is almost surely purely atomic: the weights describe the atom sizes, while the variables $(\xi_i)$ determine their locations in $E$. Quantitative approximations for these distributions have been studied using Stein's method; see Gan and Ross \cite{gan2021stein}.\\

\noindent For the purposes of this paper, it is also useful to recall the equivalent Poissonian description of the Poisson-Dirichlet law. The distribution $PD(1)$ may be obtained by taking a Poisson point process on $(0,1]$ with intensity $x^{-1}\,dx$, conditioning on its total mass being equal to one, and then arranging its atoms in decreasing order. This description is the one that will be relevant for the attenuated Poisson-Dirichlet law appearing in the main result.\\

\noindent We now record the Poissonian interpretation of the attenuation operation introduced in Section~\ref{sec:mainresults}. This point of view is useful because the limiting object in the harmonic model is naturally obtained from a scale-invariant Poisson point process, but with a constraint on its total mass different from the one appearing in the usual Poisson-Dirichlet construction. Let $\eta$ be a Poisson point process on $(0,1)$ with intensity $x^{-1}\,dx$.\\

\noindent Let $\mathfrak q(\eta)$ denote the decreasing sequence of atoms of $\eta$, counted with multiplicity and completed with zeros when necessary. The following lemma shows that this conditional law is precisely the uniformly attenuated Poisson-Dirichlet law.

\begin{lemma}\label{lem:attenuated-PD-poisson-cloud}
Let $\eta$ be a Poisson point process on $(0,1)$ with intensity $x^{-1}\,dx$, and let
\begin{align*}
T(\eta):=\int_{(0,1)}x\,\eta(dx).
\end{align*}
Then
\begin{align*}
\mathcal L\big(\mathfrak q(\eta)\mid T(\eta)\le 1\big)
=
\mathfrak A[\theta_{\mathcal S}],
\end{align*}
where $\theta_{\mathcal S}$ denotes the law on $\mathcal S$ obtained by arranging a $PD(1)$ mass partition in decreasing order. Equivalently,
\begin{align*}
\mathfrak q(\eta)\mid \{T(\eta)\le1\}
\overset{Law}{=}
U(P_1,P_2,\dots),
\end{align*}
where $(P_1,P_2,\dots)$ has distribution $PD(1)$, $U$ is uniformly distributed on $(0,1)$, and $U$ is independent of $(P_1,P_2,\dots)$.
\end{lemma}

\begin{proof}
For the scale-invariant Poisson point process with intensity $x^{-1}\,dx$, the total mass
\begin{align*}
T(\eta)=\int_{(0,1)}x\,\eta(dx)
\end{align*}
has density
\begin{align*}
t\mapsto e^{-\gamma_{\mathrm E}}\rho(t),
\end{align*}
where $\gamma_{\mathrm E}$ is the Euler-Mascheroni constant and $\rho$ is the Dickman function. Since $\rho(t)=1$ for $0\le t\le1$, the conditional law of $T(\eta)$ given $\{T(\eta)\le1\}$ is the uniform law on $(0,1)$.\\

\noindent We next use the Poissonian construction of the Poisson-Dirichlet law. The law $PD(1)$ is obtained by taking a scale-invariant Poisson point process on $(0,1]$ with intensity $x^{-1}\,dx$, conditioning on the event that the total mass of the configuration is equal to one, and arranging its atoms in decreasing order; see \cite{MR2195574}.\\

\noindent Fix $0<t\le1$, and consider a regular conditional distribution given $T(\eta)=t$. Since the total mass is $t$, all atoms of $\eta$ belong to $(0,t]$. After the change of variables $x=ty$, the normalized configuration has total mass equal to one. Moreover, the intensity $x^{-1}\,dx$ is invariant under this scaling. Hence the conditional law of the normalized ranked configuration is the same as the law obtained by conditioning a scale-invariant Poisson point process to have total mass one. Therefore
\begin{align*}
\mathcal L\big(\mathfrak q(\eta)/t\,\big|\,T(\eta)=t\big)
=
\theta_{\mathcal S},
\end{align*}
for every $0<t\le1$. In particular, this conditional law does not depend on $t$. It follows that, under the conditioning $\{T(\eta)\le1\}$, the random sequence $ \mathfrak q(\eta)/T(\eta)$
has distribution $\theta_{\mathcal S}$ and is independent of $T(\eta)$.\\

\noindent Since, under the same conditioning, $T(\eta)$ is uniformly distributed on $(0,1)$, we obtain
\begin{align*}
\mathfrak q(\eta)\mid \{T(\eta)\le1\}
\overset{d}{=}
U(P_1,P_2,\dots),
\end{align*}
where $U\sim \mathrm{Unif}(0,1)$ is independent of $(P_1,P_2,\dots)\sim PD(1)$. By the definition of the attenuation operator $\mathfrak A$, this is precisely $\mathfrak A[\theta_{\mathcal S}].$ The proof is complete.
\end{proof}

\section{Proof of Theorem \ref{theorem:main}}\label{sec:8}

\noindent We begin by fixing a notation for conditioning. Let $(E,\mathcal E)$ be a measurable space, let $\gamma$ be an element of $\mathcal M_1(E)$, and let $A\in\mathcal E$ be such that $\gamma[A]>0$. We denote by $\mathfrak C_A[\gamma]$ the probability measure on \(E\) defined by
\begin{align*}
\mathfrak C_A[\gamma][B]
:=
\frac{\gamma[B\cap A]}{\gamma[A]},
\end{align*}
for $B\in\mathcal E$. When \(E=\mathbf N(\mathbbm X)\), we shall use the notation
\begin{align}\label{eq:Aodef}
A^o
:=
\left\{
\eta\in\mathbf N(\mathbbm X)
\ ;\
\int_{\mathbbm X}x\,\eta(dx)\le1
\right\}.
\end{align}

\noindent Let the notation introduced in Section \ref{se:repofharmonicprimefact} prevail. In particular, let $(\xi_p\ ;\ p\in\mathcal P)$ and $A_n$ be as in Proposition \ref{p:div_H}. Define the random integer 
\begin{align*}
I_n := \prod_{p\le n} p^{\xi_p}.
\end{align*}
Then, by Proposition \ref{p:div_H},
\begin{align*}
\mathcal{L}(H_n)=\mathcal{L}(I_n\mid A_n).
\end{align*}
We introduce two point processes. The first one records the full geometric multiplicities:
\begin{align*}
\mathcal X_n^{\mathrm{geo}}
:=
\sum_{p\in\mathcal P_n}
\xi_p\,
\delta_{\frac{\log p}{\log n}}.
\end{align*}
The second one records only the distinct prime divisors:
\begin{align*}
\mathcal X_n^{0}
:=
\sum_{p\in\mathcal P_n}
\mathbbm 1_{\{\xi_p\ge1\}}\,
\delta_{\frac{\log p}{\log n}}.
\end{align*}
Notice that
\begin{align*}
A_n
=
\left\{
\frac{1}{\log n}
\sum_{p\le n}\xi_p\log p\le1
\right\}
=
\{\mathcal X_n^{\mathrm{geo}}\in A^o\}.
\end{align*}
On the other hand, using the fact that 
\begin{align*}
\mathfrak{P}(I_n)=\{p\le n:\xi_p\ge1\},
\end{align*}
we have
\begin{align}\label{Tpushfwardexpression}
T^n_{\#}\big[\iota\circ \mathfrak{P}(I_n)\big]
=
\sum_{p\in\mathcal{P}_n}\mathbf{1}_{\{\xi_p\ge1\}}\,
\delta_{\frac{\log p}{\log n}}
=
\mathcal X_n^0.	
\end{align}
Consequently,
\begin{align}\label{eq:trueconditionedlaw}
\mathcal{L}\!\left(T^n_{\#}\big[\iota\circ \mathfrak{P}(H_n)\big]\right)
=
\mathcal{L}\!\left(\mathcal X_n^0\,\middle|\, \mathcal X_n^{\mathrm{geo}}\in A^o\right).
\end{align}

\noindent Our strategy exploits the representation recalled in Section~\ref{se:repofharmonicprimefact}: the random object of interest can be written as a functional of a family of independent random variables, conditioned on a single global constraint. We proceed in two steps. First, we establish quantitative approximation results at the unconditioned level. These approximations are obtained by constructing a sequence of $\mathbf N(\mathbbm X)$-valued random variables
\begin{align*}
\mathcal X_n^0,\mathcal X_n^1,\mathcal X_n^2,\mathcal X_n^3
\end{align*}
on a common probability space, where $\mathcal X_n^0$ is given by \eqref{Tpushfwardexpression} and $\mathcal X_n^3$ is a Poisson point process on $(0,1]$ with intensity \(x^{-1}\,dx\). We shall prove that, for each \(i=0,1,2\),
\begin{align}\label{eq:chain-target-estimate}
\mathcal E_{A^o}[\mathcal X_n^i,\mathcal X_n^{i+1}]
+
\E\big[\mathfrak d(\mathcal X_n^i,\mathcal X_n^{i+1})\big]
\le
C\,\frac{(\log\log n)^{\mathrm e-\frac32}}{\log n},
\end{align}
up to harmless changes in the value of the constant \(C\). In addition, we need one extra discrepancy estimate at the beginning of the chain, because the true conditioning event is encoded by \(\mathcal X_n^{\mathrm{geo}}\), whereas the point process whose conditional law we study is \(\mathcal X_n^0\). This is the content of the following lemma.

\begin{lemma}\label{lem:true-conditioning-vs-distinct}
There exists a constant \(C>0\) such that, for all \(n\) sufficiently large,
\begin{align*}
\mathcal E_{A^o}[\mathcal X_n^{\mathrm{geo}},\mathcal X_n^0]
\le
C\frac{\log\log n}{\log n}.
\end{align*}
In particular,
\begin{align*}
\mathcal E_{A^o}[\mathcal X_n^{\mathrm{geo}},\mathcal X_n^0]
\le
C\frac{(\log\log n)^{\mathrm e-\frac32}}{\log n}.
\end{align*}
\end{lemma}

\begin{proof}
Set
\begin{align*}
S_n^{\mathrm{geo}}
:=
\int_{\mathbbm X}x\,\mathcal X_n^{\mathrm{geo}}(dx)
=
\frac{1}{\log n}
\sum_{p\in\mathcal P_n}\xi_p\log p
\end{align*}
and
\begin{align*}
S_n^0
:=
\int_{\mathbbm X}x\,\mathcal X_n^0(dx)
=
\frac{1}{\log n}
\sum_{p\in\mathcal P_n}
\mathbbm 1_{\{\xi_p\ge1\}}\log p.
\end{align*}
Since \(S_n^0\le S_n^{\mathrm{geo}}\), we have
\begin{align*}
\mathcal E_{A^o}[\mathcal X_n^{\mathrm{geo}},\mathcal X_n^0]
=
\Pb[S_n^0\le1<S_n^{\mathrm{geo}}].
\end{align*}
Define
\begin{align*}
R_n
:=
S_n^{\mathrm{geo}}-S_n^0
=
\frac{1}{\log n}
\sum_{p\in\mathcal P_n}
\big(\xi_p-\mathbbm 1_{\{\xi_p\ge1\}}\big)\log p.
\end{align*}
For every \(\tau>0\),
\begin{align*}
\{S_n^0\le1<S_n^{\mathrm{geo}}\}
\subseteq
\{R_n>\tau\}
\cup
\{1<S_n^{\mathrm{geo}}\le1+\tau\}.
\end{align*}
Hence
\begin{align}\label{eq:initial-discrepancy-split}
\mathcal E_{A^o}[\mathcal X_n^{\mathrm{geo}},\mathcal X_n^0]
\le
\Pb[R_n>\tau]
+
\Pb[1<S_n^{\mathrm{geo}}\le1+\tau].
\end{align}

\noindent We first control \(R_n\). Let
\begin{align*}
Y_p
:=
\xi_p-\mathbbm 1_{\{\xi_p\ge1\}}.
\end{align*}
Then \(Y_p=0\) when \(\xi_p=0\) or \(\xi_p=1\), while \(Y_p=k-1\) when \(\xi_p=k\ge2\). Fix \(0<s<1\). Since
\begin{align*}
\PP[\xi_p=k]=(1-p^{-1})p^{-k},
\end{align*}
and \(Y_p=0\) unless \(\xi_p\ge2\), there exists a constant \(C_s>0\), depending only on \(s\), such that
\begin{align*}
\E[p^{sY_p}]
\le
\exp\big\{ C_s p^{s-2}\big\}.
\end{align*}
By exponential Chebyshev and independence,
\begin{align*}
\Pb[R_n>\tau]
&\le
\exp\{-s\tau\log n\}
\prod_{p\in\mathcal P_n}
\E[p^{sY_p}]
\le
\exp\{-s\tau\log n\}
\exp\left\{
C_s\sum_{p\in\mathcal P} p^{s-2}
\right\}.
\end{align*}
Thus, after increasing \(C_s\) if necessary,
\begin{align}\label{eq:Rn-tail-geo-zero}
\Pb[R_n>\tau]
\le
C_s e^{-s\tau\log n}.
\end{align}

\noindent We now control the concentration term. By Lemma~\ref{lem:kolmmain} and the boundedness of \(\rho\) on compact intervals, there exists a universal constant \(C>0\) such that
\begin{align}\label{eq:Sgeo-small-window}
\Pb[1<S_n^{\mathrm{geo}}\le1+\tau]
\le
\frac{C}{\log n}
+
C\tau.
\end{align}
Choose
\begin{align*}
\tau
=
K\frac{\log\log n}{\log n},
\end{align*}
where \(K>0\) is large enough so that \(sK>2\). Combining \eqref{eq:initial-discrepancy-split}, \eqref{eq:Rn-tail-geo-zero}, and \eqref{eq:Sgeo-small-window}, we obtain
\begin{align*}
\mathcal E_{A^o}[\mathcal X_n^{\mathrm{geo}},\mathcal X_n^0]
\le
C\frac{\log\log n}{\log n}.
\end{align*}
Since \(\mathrm e-\frac32>1\), the second estimate follows. The proof is complete.
\end{proof}

\noindent The estimates \eqref{eq:chain-target-estimate}, together with Lemma~\ref{lem:true-conditioning-vs-distinct} and the triangle inequality for \(\mathcal E_{A^o}\), imply that
\begin{align}\label{eq:total-event-discrepancy}
\mathcal E_{A^o}[\mathcal X_n^{\mathrm{geo}},\mathcal X_n^3]
\le
C\,\frac{(\log\log n)^{\mathrm e-\frac32}}{\log n}.
\end{align}
Moreover, the metric part of \eqref{eq:chain-target-estimate} gives
\begin{align}\label{eq:total-metric-discrepancy}
\E[\mathfrak d(\mathcal X_n^0,\mathcal X_n^3)]
\le
C\,\frac{(\log\log n)^{\mathrm e-\frac32}}{\log n}.
\end{align}
The second step of the proof, carried out in Section~\ref{sec:conditioningstep}, transfers \eqref{eq:total-event-discrepancy} and \eqref{eq:total-metric-discrepancy} to the conditioned laws. The limiting conditioned law is identified as follows. If \(\eta\) is a Poisson point process on \((0,1]\) with intensity \(x^{-1}\,dx\), then Lemma~\ref{lem:attenuated-PD-poisson-cloud} gives
\begin{align*}
\mathcal L\big(\mathfrak q(\eta)\mid T(\eta)\le1\big)
=
\mathfrak A[\theta_{\mathcal S}],
\qquad
T(\eta):=\int_{(0,1]}x\,\eta(dx).
\end{align*}
Thus, after passing from point measures to their ordered atom sequences, conditioning the limiting Poisson cloud on \(A^o\) gives precisely the uniformly attenuated Poisson-Dirichlet law appearing in Theorem~\ref{theorem:main}.\\

\noindent The analysis of the unconditioned approximation chain begins with a Poissonization step, embedding the right-hand side of \eqref{Tpushfwardexpression} into a Poisson point process representation. This allows us to approximate it by a simpler Poisson process with explicit error bounds.

\subsection{Poissonization  step}\label{sec:poissonizationsteop}
Define the ambient space 
\[
\mathbbm{Y} := \{(p,k): p\in \mathcal{P},\ k\in\mathbb{N}\},
\]
and equip this space with the counting $\sigma$-field. Let $\ell$ be the measure on $\mathbbm{Y}$ given by
\[
\ell[p,k]:=\ell[\{(p,k)\}]:=\frac{1}{k p^k},
\]
with $(p,k)$ in $ \mathbbm{Y}$. Let $N$ be a Poisson point process on $ \mathbbm{Y}$ with intensity $\ell$.  Defining the kernel
\[
f_p:\mathbbm{Y}\to\R_+,
\qquad
f_p(q,k):=k\,\mathbbm{1}_{\{q=p\}},
\]
we have the Poisson integral representation
\[
\xi_p = \int_{\mathbbm{Y}} f_p(y)\,N(dy).
\]
Define then the random variables $\tilde{\xi}_{p}:=N[\{(p,1)\}]$. Then \(\tilde \xi_p\) is Poisson distributed with intensity $1/p$, and the variables \(\{\tilde \xi_p\ ;\ p\in\mathcal P\}\) are independent. We also define the point processes
\begin{align}\label{eq:newgoalpoitmeasure}
\mathcal{X}_n^{0}
  :=\sum_{p\in\mathcal{P}_n} \mathbbm{1}_{\{\xi_p\geq 1\}}
\delta_{\frac{\log p}{\log n}}
\quad\quad\quad\quad
\mathcal{X}_n^{1}
  :=\sum_{p\in\mathcal{P}_n}\tilde{\xi}_p
\delta_{\frac{\log p}{\log n}}.	
\end{align}
Observe that $\mathcal{X}_n^{0}$ is a simple point process with intensity measure 
\begin{align*}
\mu_n^{0}
  &=\sum_{p\in\mathcal{P}_n}\frac{1}{p}\delta_{\frac{\log(p)}{\log(n)}}.	
\end{align*}
We will show that   
\begin{align}\label{eq:goalpoissonization}
\E[\mathfrak{d}(\mathcal{X}_n^{0},\mathcal{X}_n^{1})]
  &\leq \kappa_1/\log(n),
\end{align}
where 
\begin{align*}
\kappa_1
  &:=2\left(\frac{\log 2}{4}+\frac{\log 3}{9}+\frac{\log 4+1}{4}\right).	
\end{align*}
To do this, we  denote by $\varepsilon_n$ the left hand side of \eqref{eq:goalpoissonization}. Observe that by \eqref{eq:acouplingbounddfrak}, 

\begin{align*}
\varepsilon_n
  &\leq
\sum_{p\in\mathcal{P}_n}
\mathbb E\!\left[|\mathbbm{1}_{\{\xi_p\geq 1\}}-\tilde{\xi}_p|\right]
\frac{\log(p)}{\log(n)}.
\end{align*}
Since $\tilde{\xi}_p=N[\{(p,1)\}]$ and 
$$\mathbbm 1_{\{\xi_p\ge1\}}=\mathbbm 1_{\{N(\{(p,1)\})+\sum_{k\ge2}N(\{(p,k)\})\ge1\}},$$ 
we have the bound
\begin{align*}
\big|\mathbbm{1}_{\{\xi_p\ge1\}}-\tilde{\xi}_p\big|
\le
\big(\tilde{\xi}_p-1\big)_+
+
\sum_{k\ge2} N[\{(p,k)\}].
\end{align*}
Taking expectations, we obtain
\begin{align*}
\varepsilon_n
\le
\sum_{p\in\mathcal P_n}
\big(
\mathbb E\big[(\tilde{\xi}_p-1)_+\big]
+
\sum_{k\ge2}\ell[p,k]
\big)
\frac{\log p}{\log n}.
\end{align*}
Since $\tilde{\xi}_p\sim\mathrm{Poisson}(\ell[p,1])$, we have
\begin{align*}
\mathbb E\big[(\tilde{\xi}_p-1)_+\big]
=
\ell[p,1]
-
\big(1-e^{-\ell[p,1]}\big),
\end{align*}
and therefore
\begin{align*}
\varepsilon_n
\le
\sum_{p\in\mathcal P_n}
\left(
\ell[p,1]
-
1
+
e^{-\ell[p,1]}
+
\sum_{k\ge2}\ell[p,k]
\right)
\frac{\log p}{\log n}.
\end{align*}
Recall that $\ell[p,k]=\frac{1}{k p^k}$, so we have the identities $\ell[p,1]=\frac{1}{p}$ and
\begin{align*}
\sum_{k\ge2}\ell[p,k]
=
\sum_{k\ge2}\frac{1}{k p^k}
=
\sum_{k\ge1}\frac{1}{k p^k}-\frac{1}{p}
=
-\log\!\Big(1-\frac{1}{p}\Big)-\frac{1}{p}.
\end{align*}
Consequently, 
\begin{align*}
\varepsilon_n
\le
\sum_{p\in\mathcal P_n}
\left(
-1
+
e^{-1/p}
-
\log\!\Big(1-\frac{1}{p}\Big)
\right)
\frac{\log p}{\log n}.
\end{align*}
Using the Taylor expansions at $x=0$, one can show that
\begin{align*}
0\le -1+e^{-1/p}-\log\!\Big(1-\frac1p\Big)\le \frac{2}{p^2},
\end{align*}
yielding the inequality 
\begin{align*}
\varepsilon_n
\leq
\frac{2}{\log n}\sum_{p\in\mathcal P_n}\frac{\log p}{p^{2}}.
\end{align*}
Removing the first two prime terms and then enlarging the sum from primes to integers and applying an integral comparison yields
\begin{align*}
\varepsilon_n
\le
\frac{2}{\log n}
\left(
\frac{\log 2}{4}+\frac{\log 3}{9}+\frac{\log 4+1}{4}
\right).
\end{align*}

\noindent To estimate $\mathcal E_{A^o}[\mathcal X_n^{0},\mathcal X_n^{1}]$, we define
\[
S_n^{1}
:=
\frac{1}{\log n}
\sum_{p\in\mathcal P_n}\tilde\xi_p\log p
\quad \quad \quad \quad \quad \quad \quad \quad 
D_n
:=
\frac{1}{\log n}
\sum_{p\in\mathcal P_n}
\big|
\mathbbm 1_{\{\xi_p\ge1\}}-\tilde\xi_p
\big|\,\log p.
\]
Recall the set \(A^o\) defined in \eqref{eq:Aodef}. One can easily check that 
\[
\{\mathcal X_n^{0}\in A^{o}\}\Delta\{\mathcal X_n^{1}\in A^{o}\}
\subseteq
\big\{
|S_n^{1}-1|
\le
D_n
\big\},
\]
and therefore
\begin{align*}
\mathcal E_{A^o}[\mathcal X_n^{0},\mathcal X_n^{1}]
\le
\Pb\big[|S_n^{1}-1|\le D_n\big].
\end{align*}
We split the right-hand side as
\begin{align*}
\Pb\big[|S_n^{1}-1|\le D_n\big]
&\le
\varepsilon_n^{1}(\tau)+\varepsilon_n^{2}(\tau),
\end{align*}
where 
\begin{align*}
\varepsilon_n^{1}(\tau)
  :=\Pb[D_n>\tau],
  \quad\quad\quad\quad\quad\quad
\varepsilon_n^{2}(\tau)
  :=\Pb\big[|S_n^{1}-1|\le D_n,\ D_n\le \tau\big].
\end{align*}

\noindent We now derive a bound for the tail probability of $D_n$ using an exponential-Chebyshev argument. For $p$ in \(\mathcal P_n\), define
\[
Y_p
:=
\big|\mathbbm 1_{\{\xi_p\ge1\}}-\tilde\xi_p\big|,
\qquad a_p:=N[\{(p,1)\}],
\qquad
b_p:=\sum_{k\ge2}N[\{(p,k)\}].
\]
The pairs \((a_p,b_p)\) are independent, and for each fixed \(p\), the random variables \(a_p\) and \(b_p\) are independent. With these definitions, the event $\{b_p=0\}$ coincides with the event that no atom of the form $(p,k)$ with $k\ge2$ appears. Consequently, after some simple case analysis, we can show that
\[
Y_p
=
\begin{cases}
(a_p-1)_+, & \text{if }\  b_p=0,\\[0.3em]
(a_p-1)_+ + \mathbbm 1_{\{a_p=0\}}, & \text{if }\  b_p\ge1.
\end{cases}
\]
A direct computation using the independence of \(a_p\) and \(b_p\), together with the preceding case distinction, gives
\begin{align*}
\E[p^{sY_p}]
&=
\E[p^{s(a_p-1)_+}]
+
(1-e^{-\lambda_p})e^{-1/p}(p^s-1),
\end{align*}
where
\begin{align*}
\lambda_p:=\sum_{k\ge2}\frac{1}{kp^k}.
\end{align*}
Using \(p^{s-1}\le1\), one obtains
\begin{align}\label{eq:boundexpoYp}
\E[e^{sY_p \log p}]
=
\E[p^{sY_p}]
\le \exp\!\left(\frac{\mathrm e-\frac 32}{p^{2-s}}\right),
\end{align}
for \(p \ge 2\) and \(0 \le s < 1\). Now fix $0<s<1$. By exponential Chebyshev,
\begin{align*}
\Pb[D_n>\tau]
&\le
e^{-s\tau\log n}
\E \big[
\exp \big\{
s\sum_{p\in\mathcal P_n}Y_p\log p
\big\}
\big]
=
e^{-s\tau\log n}
\prod_{p\in\mathcal P_n}
\E\left[e^{sY_p\log p}\right].
\end{align*}

By \eqref{eq:boundexpoYp}, we then have that
\begin{align*}
\Pb[D_n>\tau]
&\le
e^{-s\tau\log n}
\prod_{p\in\mathcal P_n}
\exp\!\left\{ \frac{e-\frac 32}{p^{2-s}}\right\}
=
\exp\left\{
-s\tau\log n
+
\sum_{p\in\mathcal P_n}\frac{e-\frac 32}{p^{2-s}}
\right\}.
\end{align*}
\noindent We now derive an explicit bound for the sum in the right-hand side. Define
\[
A(x):=\sum_{p\in\mathcal{P}}\mathbbm{1}_{\{p\le x\}}\frac{1}{p}.
\]
Then, by summation by parts, for $n>3$, 
\begin{align*}
\sum_{p\le n}\frac{1}{p^{2-s}}=\frac{1}{2^{2-s}}+\sum_{3\leq p\le n}\frac{1}{p^{2-s}}
&=
\frac{1}{2^{2-s}}+
\frac{A(n)}{n^{1-s}}-\frac{1}{3^{2-s}}
+
(1-s)\int_3^n \frac{A(x)}{x^{2-s}}\,dx.
\end{align*}
Using \eqref{Secondmertens}, we have that $A(x)\le 3+\log\log x$ for $x\ge 3$, which in combination with the identity above, yields
\begin{align*}
\sum_{p\le n}\frac{1}{p^{2-s}}
&\le
\frac{1}{2^{2-s}}
+
\frac{3+\log\log n}{n^{1-s}}
+
(1-s)\int_3^n \frac{3+\log\log x}{x^{2-s}}\,dx.
\end{align*}
To estimate the integral, we apply the change of variables $u=(1-s)\log x$, which yields
\begin{align*}
(1-s)\int_3^n \frac{3+\log\log x}{x^{2-s}}\,dx
&=
\int_{(1-s)\log 3}^{(1-s)\log n}
\left( 3+\log u+\log \left( \frac{ 1} {1-s} \right) \right) e^{-u}\,du.
\end{align*}
From here it easily follows that
\begin{align*}
(1-s)\int_3^n \frac{3+\log\log x}{x^{2-s}}\,dx
&\le
3+\log \left( \frac{ 1} {1-s} \right) +2
=
5+\log \left( \frac{ 1} {1-s} \right) .
\end{align*}
Combining the above estimates, we obtain
\begin{align*}
\sum_{p\le n}\frac{1}{p^{2-s}}
&\le
\frac{1}{2^{2-s}}
+
\frac{3+\log\log n}{n^{1-s}}
+
5+\log \left( \frac{ 1} {1-s} \right).
\end{align*}
It then follows that 
\begin{align}\label{eq:boundDnED}
\Pb[D_n>\tau]
&\le
\exp \left \{
-s\tau\log n
+
\left(e-\frac 32 \right) \left( 2^{s-2}
+
\frac{3+\log\log n}{n^{1-s}}
+
5+\log\!\left( \frac{1}{1-s}\right) \right)
\right \}.
\end{align}
Take, for $n$ sufficiently large,
\[
\tau:=\frac{\log\log n}{\log n}
\qquad\text{and}\qquad
s:=1-\frac{e-\frac 32}{\log\log n},
\]
so that \eqref{eq:boundDnED} yields
\begin{align*}
\Pb[D_n>\tau]
&\le
\frac{C}{\log n}
\exp\left\{
\left( e-\frac 32 \right)
\left(
2^{-1-\frac{e-3/2}{\log\log n}}
+
\frac{3+\log\log n}{n^{\frac{e- 3/2}{\log\log n}}}
+ \log\log\log n\right)
\right\}.
\end{align*}
By elementary computations, we can show that the term 
\begin{align*}
2^{-1-\frac{e-3/2}{\log\log n}}
+
\frac{3+\log\log n}{n^{\frac{e- 3/2}{ \log\log n }}}
\end{align*}
is bounded uniformly for $n$ sufficiently large, and thus, we can guarantee the existence of a constant $C>0$, such that 
\begin{align*}
\Pb[D_n>\tau]
&\le
\frac{C}{\log n}
\exp\left\{ \left( e-\frac 32 \right)\log\log\log n
\right\}.
\end{align*}
After suitable algebraic manipulations, we get 
\begin{align*}
\varepsilon_n^{1}(\tau)
&\le
\frac{C}{\log n}\,(\log\log n)^{e-\frac 32},
\end{align*}
for some universal constant \(C>0\).\\

\noindent To control $\varepsilon_n^{2}$, define
\[
S_n^{\mathrm{geo}}
:=
\frac{1}{\log n}
\sum_{p\in\mathcal P_n}\xi_p\log p
\]
and
\[
R_n
:=
S_n^{\mathrm{geo}}-S_n^{1}
=
\frac{1}{\log n}
\sum_{p\in\mathcal P_n}
\sum_{k\ge2}k\,N[\{(p,k)\}]\,\log p.
\]
Then, for every $\delta>0$,
\[
\big\{|S_n^{1}-1|\le \tau\big\}
\subseteq
\big\{|S_n^{\mathrm{geo}}-1|\le \tau+\delta\big\}
\cup
\{R_n>\delta\}.
\]
Consequently,
\[
\varepsilon_n^{2}
\le
\Pb\big[|S_n^{\mathrm{geo}}-1|\le \tau+\delta\big]
+
\Pb[R_n>\delta].
\]
By Lemma \ref{lem:kolmmain} and the boundedness of $\rho$ on compact intervals, there exists a universal constant \(C>0\) such that
\[
\Pb\big[|S_n^{\mathrm{geo}}-1|\le \tau+\delta\big]
\le
\frac{C}{\log n}+C(\tau+\delta).
\]
Moreover, fix $0<t<1/2$. Exponential Chebyshev gives
\begin{align*}
\Pb[R_n>\delta]
&\le
e^{-t\delta\log n}
\E\big[
\exp\big\{
t\sum_{p\in\mathcal P_n}
\sum_{k\ge2}k\,N[\{(p,k)\}]\,\log p
\big\}
\big] \\
&=
e^{-t\delta\log n}
\exp\big\{
\sum_{p\in\mathcal P_n}\sum_{k\ge2}
\frac{p^{tk}-1}{kp^k}
\big\} 
\le
C_t n^{-t\delta},
\end{align*}
where
\[
C_t
:=
\exp\big\{
\sum_{p\in\mathcal P}\sum_{k\ge2}
\frac{1}{k p^{k(1-t)}}
\big\}
<\infty.
\]
Take $\delta_n$ as the unique solution of
\[
\delta_n=n^{-t\delta_n}.
\]
Let $V$ be the inverse of $u\mapsto ue^u$. Then we can write $\delta_n$ as 
\[
\delta_n
=
\frac{V(t\log n)}{t\log n}.
\]
One can easily check that  
$$\limsup_{x\rightarrow\infty}V(x)/\log x<\infty,$$ 
which yields the inequality
\[
\delta_n\le C\frac{\log\log n}{\log n},
\]
for some constant  $C>0$. Therefore
\begin{align*}
\varepsilon_n^{2}
&\le
\frac{C}{\log n}+C\tau+(C+C_t)\delta_n
\le
C\frac{\log\log n}{\log n}.
\end{align*}
Combining the estimates for \(\varepsilon_n^{1}\) and \(\varepsilon_n^{2}\), we conclude that
\begin{align}\label{ineq:Deltabound}
\mathcal E_{A^o}[\mathcal X_n^{0},\mathcal X_n^{1}]
&\le
\frac{C}{\log n}\,(\log\log n)^{e-\frac 32},
\end{align}
for some universal constant \(C>0\).

\subsection{Simple point process approximation for \(\mathcal X_n^1\)}

The goal of this section is to construct an \(\mathbf N(\mathbbm X)\)-valued Poisson random variable \(\mathcal X_n^2\) which is close to \(\mathcal X_n^1\), and whose control measure is a shifted absolutely continuous approximation of the Dickman control measure, with a small correction in the underlying control measure. Set
\begin{align*}
x_n:=\frac{\log 2}{\log n},
\end{align*}
and consider the truncated Dickman control measure
\begin{align*}
\nu_n(dx)
:=
\frac{1}{x}\mathbbm 1_{(x_n,1]}(x)\,dx.
\end{align*}
For \(0<t\le1\), the control measure of \(\mathcal X_n^1\) satisfies
\begin{align*}
\mu_n^1[(t,1]]
&=
\sum_{p\in\mathcal P_n}\frac{1}{p}\,
\mathbbm 1_{\{\frac{\log p}{\log n}>t\}}
=
\sum_{p\in\mathcal P}
\mathbbm 1_{(n^t,n]}(p)\frac{1}{p}.
\end{align*}
For \(x>0\), define
\begin{align*}
M(x):=\sum_{p\le x}\frac{1}{p}.
\end{align*}
By Mertens' formula, after increasing the constant if necessary, there exists \(B\in\mathbb R\) such that, for all \(x\ge2\),
\begin{align}\label{eq:mertensrefffort}
\left|
M(x)-\log\log x-B
\right|
\le
\frac{4}{\log x}.
\end{align}
Therefore, for \(t\ge x_n\),
\begin{align*}
\left|
\mu_n^1[(t,1]]+\log t
\right|
&=
\left|
M(n)-M(n^t)+\log t
\right|
\le
\frac{4}{\log n}
+
\frac{4}{t\log n}.
\end{align*}

\noindent We now introduce the absolutely continuous approximation that will be used to dominate the prime control measure at the level of upper tails. The estimate above shows that the discrepancy between \(\mu_n^1[(t,1]]\) and \(-\log t\) contains a term of order \((t\log n)^{-1}\). To compensate for this term, we shift the truncated Dickman measure slightly to the right in the following way. Define
\begin{align*}
\delta_n:=\frac{K_0}{\log n},
\qquad
y_n:=x_n+\delta_n,
\end{align*}
where \(K_0>0\) is a fixed numerical constant, independent of \(n\), whose value will be specified later. Define
\begin{align*}
\Phi_n(x):=(x+\delta_n)\wedge1.
\end{align*}
The pushforward \((\Phi_n)_{\#}\nu_n\) is therefore a right-shifted version of the truncated Dickman control measure. This shift increases the upper tails by an amount of order
\begin{align*}
-\log(t-\delta_n)+\log t
\sim
\frac{\delta_n}{t}
=
\frac{K_0}{t\log n},
\end{align*}
which is precisely the scale needed to absorb the tail error obtained from Mertens' formula.\\

\noindent
The right shift above will be used to control the upper-tail comparison for thresholds \(t\ge y_n\). For the remaining thresholds, those below \(y_n\), we add a fixed amount of mass at the left edge of the shifted support. Let \(a>0\) be another fixed numerical constant, independent of \(n\), whose value will be specified later, and define
\begin{align*}
\mu_n^2
:=
(\Phi_n)_{\#}\nu_n
+
a\,\delta_{y_n}.
\end{align*}
The atom at \(y_n\) contributes to every upper tail \((t,1]\) with \(t<y_n\), and this is the extra mass that will be used to handle the small-threshold region. Since \(y_n\) is of order \(1/\log n\), this correction has small cost in the metric \(\mathfrak d\). Notice also that the pushforward \((\Phi_n)_{\#}\nu_n\) may create an atom at \(1\), coming from the interval \((1-\delta_n,1]\), but this atom has intensity of order \(1/\log n\).\\

\noindent
We now introduce a small auxiliary modification of the control measure of \(\mathcal X_n^1\). This modification is used only to match the total mass of \(\mu_n^2\), so that Lemma~\ref{Propcouplmonotone} can be applied without any residual Poisson component. Define
\begin{align*}
b_n
:=
\mu_n^2((0,1])
-
\mu_n^1((0,1])
=
-\log x_n+a-M(n),
\end{align*}
and set
\begin{align*}
\widehat\mu_n^1
:=
\mu_n^1+b_n\delta_{x_n}.
\end{align*}
Since
\begin{align*}
M(n)+\log x_n
=
B+\log\log 2+O\left(\frac1{\log n}\right),
\end{align*}
as follows from \eqref{eq:mertensrefffort}, we may choose the fixed constant \(a\) so that
\begin{align*}
0\le b_n\le C
\end{align*}
for all \(n\) sufficiently large. By construction,
\begin{align*}
\widehat\mu_n^1((0,1])
=
\mu_n^2((0,1]).
\end{align*}

\noindent
We next check the upper-tail order. We claim that, after fixing \(K_0\) and \(a\) appropriately,
\begin{align*}
\widehat\mu_n^1[(t,1]]
\le
\mu_n^2[(t,1]],
\qquad 0<t\le1.
\end{align*}
Indeed, let first \(y_n\le t<1\). Then neither the atom \(b_n\delta_{x_n}\) nor the atom \(a\delta_{y_n}\) contributes to the tail \((t,1]\). Therefore
\begin{align*}
\widehat\mu_n^1[(t,1]]
=
\mu_n^1[(t,1]]
&\le
-\log t
+
\frac{4}{\log n}
+
\frac{4}{t\log n},
\end{align*}
whereas
\begin{align*}
\mu_n^2[(t,1]]
&=
\nu_n[(t-\delta_n,1]]
=
-\log(t-\delta_n).
\end{align*}
Since
\begin{align*}
-\log(t-\delta_n)+\log t
&=
-\log\left(1-\frac{\delta_n}{t}\right)
\ge
\frac{\delta_n}{t}
=
\frac{K_0}{t\log n},
\end{align*}
and
\begin{align*}
\frac{4}{\log n}
+
\frac{4}{t\log n}
\le
\frac{8}{t\log n},
\end{align*}
we obtain
\begin{align*}
\widehat\mu_n^1[(t,1]]
\le
\mu_n^2[(t,1]],
\qquad y_n\le t<1,
\end{align*}
provided \(K_0\ge8\). The case \(t=1\) is trivial.\\

\noindent
If \(x_n\le t<y_n\), then the atom \(b_n\delta_{x_n}\) does not contribute to \(\widehat\mu_n^1[(t,1]]\), while the atom \(a\delta_{y_n}\) contributes to \(\mu_n^2[(t,1]]\). Hence
\begin{align*}
\mu_n^2[(t,1]]
\ge
\nu_n[(x_n,1]]
+
a
=
-\log x_n+a.
\end{align*}
On the other hand, by \eqref{eq:mertensrefffort} and the bound \(t\ge x_n\),
\begin{align*}
\widehat\mu_n^1[(t,1]]
=
\mu_n^1[(t,1]]
&\le
-\log t
+
\frac{4}{\log n}
+
\frac{4}{t\log n}
\\
&\le
-\log x_n
+
\frac{4}{\log n}
+
\frac{4}{\log 2}.
\end{align*}
Thus, after choosing \(a\) large enough, we get
\begin{align*}
\widehat\mu_n^1[(t,1]]
\le
\mu_n^2[(t,1]],
\qquad x_n\le t<y_n.
\end{align*}

\noindent Finally, if \(0<t<x_n\), then both \(\widehat\mu_n^1\) and \(\mu_n^2\) have no mass in \((0,t]\). Hence
\begin{align*}
\widehat\mu_n^1[(t,1]]
=
\widehat\mu_n^1((0,1])
=
\mu_n^2((0,1])
=
\mu_n^2[(t,1]].
\end{align*}
Therefore
\begin{align*}
\widehat\mu_n^1\preceq\mu_n^2.
\end{align*}

\noindent
By Lemma~\ref{Propcouplmonotone}, applied to the measures \(\widehat\mu_n^1\) and \(\mu_n^2\), there exist Poisson point processes \(\widehat{\mathcal X}_n^1\) and \(\mathcal X_n^2\), with control measures \(\widehat\mu_n^1\) and \(\mu_n^2\), respectively, coupled monotonically. In particular,
\begin{align*}
\int_{\mathbbm X}x\,\widehat{\mathcal X}_n^1(dx)
\le
\int_{\mathbbm X}x\,\mathcal X_n^2(dx)
\end{align*}
almost surely.\\

\noindent
Finally, we realize \(\widehat{\mathcal X}_n^1\) as a small left-edge perturbation of the original process:
\begin{align*}
\widehat{\mathcal X}_n^1
=
\mathcal X_n^1
+
Z_n\delta_{x_n},
\end{align*}
where \(Z_n\) is independent of \(\mathcal X_n^1\) and has Poisson distribution with mean \(b_n\). Since \(b_n\le C\), we have
\begin{align}\label{eq:X1-X1hat-cost}
\E[\mathfrak d(\mathcal X_n^1,\widehat{\mathcal X}_n^1)]
\le
x_n\E[Z_n]
\le
\frac{C}{\log n}.
\end{align}

\noindent We now estimate the cost between \(\widehat{\mathcal X}_n^1\) and \(\mathcal X_n^2\). Since the coupling is monotone and the total masses of \(\widehat\mu_n^1\) and \(\mu_n^2\) agree, the expected cost is bounded by the difference of the first moments:
\begin{align*}
\E[\mathfrak d(\widehat{\mathcal X}_n^1,\mathcal X_n^2)]
\le
\int_{\mathbbm X}x\,(\mu_n^2-\widehat\mu_n^1)(dx).
\end{align*}
We estimate the right hand side. First,
\begin{align*}
\int_{\mathbbm X}x\,\mu_n^1(dx)
=
\frac{1}{\log n}
\sum_{p\in\mathcal P_n}\frac{\log p}{p}.
\end{align*}
By \eqref{Firstmertens}, for \(n\) sufficiently large,
\begin{align*}
\left|
\int_{\mathbbm X}x\,\mu_n^1(dx)-1
\right|
\le
\frac{3}{\log n}.
\end{align*}
Moreover,
\begin{align*}
\int_{\mathbbm X}x\,\widehat\mu_n^1(dx)
=
\int_{\mathbbm X}x\,\mu_n^1(dx)
+
b_nx_n.
\end{align*}
Since \(b_n\le C\), the extra term is \(O(1/\log n)\). On the other hand,
\begin{align*}
\int_{\mathbbm X}x\,(\Phi_n)_{\#}\nu_n(dx)
&=
\int_{(x_n,1]}\Phi_n(x)\frac{dx}{x}
\le
\int_{x_n}^1x\,\frac{dx}{x}
+
\delta_n\int_{x_n}^1\frac{dx}{x}
=
1-x_n
+
\delta_n\log\left(\frac1{x_n}\right),
\end{align*}
and
\begin{align*}
\int_{\mathbbm X}x\,a\delta_{y_n}(dx)
=
ay_n
\le
\frac{C}{\log n}.
\end{align*}
Combining these estimates, we obtain
\begin{align*}
\int_{\mathbbm X}x\,(\mu_n^2-\widehat\mu_n^1)(dx)
\le
C\frac{\log\log n}{\log n}.
\end{align*}
Together with \eqref{eq:X1-X1hat-cost}, this gives
\begin{align*}
\E[\mathfrak d(\mathcal X_n^1,\mathcal X_n^2)]
\le
C\frac{\log\log n}{\log n}.
\end{align*}
Since \(\mathrm e-\frac32>1\), this also implies
\begin{align*}
\E[\mathfrak d(\mathcal X_n^1,\mathcal X_n^2)]
\le
C\frac{(\log\log n)^{\mathrm e-\frac32}}{\log n}.
\end{align*}

\noindent
We now estimate \(\mathcal E_{A^o}[\mathcal X_n^2,\mathcal X_n^1]\). By the triangle inequality for \(\mathcal E_{A^o}\),
\begin{align*}
\mathcal E_{A^o}[\mathcal X_n^2,\mathcal X_n^1]
\le
\mathcal E_{A^o}[\mathcal X_n^2,\widehat{\mathcal X}_n^1]
+
\mathcal E_{A^o}[\widehat{\mathcal X}_n^1,\mathcal X_n^1].
\end{align*}

\noindent
We first handle the second term. Set
\begin{align*}
S_n^1
:=
\int_{\mathbbm X}x\,\mathcal X_n^1(dx),
\qquad
\widehat S_n^1
:=
\int_{\mathbbm X}x\,\widehat{\mathcal X}_n^1(dx).
\end{align*}
Then
\begin{align*}
\widehat S_n^1-S_n^1=x_nZ_n.
\end{align*}
For every \(r>0\),
\begin{align*}
\mathcal E_{A^o}[\widehat{\mathcal X}_n^1,\mathcal X_n^1]
&\le
\Pb[x_nZ_n>r]
+
\Pb[1-r\le S_n^1\le1].
\end{align*}

\noindent
We next handle the first term. Set
\begin{align*}
S_n^2
:=
\int_{\mathbbm X}x\,\mathcal X_n^2(dx),
\end{align*}
and define
\begin{align*}
R_n:=S_n^2-\widehat S_n^1.
\end{align*}
Under the monotone coupling,
\begin{align*}
\widehat S_n^1\le S_n^2.
\end{align*}
Hence
\begin{align*}
\mathcal E_{A^o}[\mathcal X_n^2,\widehat{\mathcal X}_n^1]
=
\Pb[\widehat S_n^1\le1<S_n^2]
=
\Pb[1-R_n\le \widehat S_n^1\le1].
\end{align*}
For every \(r>0\),
\begin{align*}
\Pb[1-R_n\le \widehat S_n^1\le1]
&\le
\Pb[R_n>r]
+
\Pb[1-r\le \widehat S_n^1\le1].
\end{align*}

\noindent
We now choose
\begin{align*}
r_n:=K_1\frac{\log\log n}{\log n},
\end{align*}
with \(K_1>0\) large enough. We claim that the inverse-tail coupling may be chosen so that matched atoms move by at most \(2\delta_n\). To see this, it is enough to prove the reverse shifted tail inequality
\begin{align}\label{eq:reverse-shifted-tail}
\mu_n^2[(t+2\delta_n,1]]
\le
\widehat\mu_n^1[(t,1]],
\end{align}
for $ 0<t<1.$ Indeed, together with \(\widehat\mu_n^1\preceq\mu_n^2\), this implies that the inverse-tail maps used in Lemma~\ref{Propcouplmonotone} differ by at most \(2\delta_n\).\\

\noindent Let us prove \eqref{eq:reverse-shifted-tail}. If \(0<t<x_n\), then
\begin{align*}
\mu_n^2[(t+2\delta_n,1]]
\le
\mu_n^2((0,1])
=
\widehat\mu_n^1((0,1])
=
\widehat\mu_n^1[(t,1]].
\end{align*}
If \(x_n\le t<1\), then \(t+2\delta_n>y_n\), so the atom \(a\delta_{y_n}\) does not contribute to \(\mu_n^2[(t+2\delta_n,1]]\). Moreover, if \(t+2\delta_n\ge1\), the left hand side is zero. Otherwise,
\begin{align*}
\mu_n^2[(t+2\delta_n,1]]
&=
\nu_n[(t+\delta_n,1]]
=
-\log(t+\delta_n).
\end{align*}
On the other hand, by \eqref{eq:mertensrefffort},
\begin{align*}
\widehat\mu_n^1[(t,1]]
=
\mu_n^1[(t,1]]
&\ge
-\log t
-
\frac{4}{\log n}
-
\frac{4}{t\log n}.
\end{align*}
After increasing \(K_0\), we have
\begin{align*}
\log\left(1+\frac{\delta_n}{t}\right)
\ge
\frac{4}{\log n}
+
\frac{4}{t\log n},
\end{align*}
for $x_n\leq t<1$. Indeed, writing \(u=t\log n\), this reduces to choosing \(K_0\) so that
\begin{align*}
\log\left(1+\frac{K_0}{u}\right)
\ge
\frac{4}{\log n}+\frac{4}{u},
\end{align*}
for $\log 2\leq u\leq \log n$. This is possible uniformly in \(n\), since the right hand side is bounded by \(8/u\) and \(u\log(1+K_0/u)\) can be made uniformly larger than \(8\) on \(u\ge\log 2\) by choosing \(K_0\) large enough. Therefore
\begin{align*}
-\log(t+\delta_n)
\le
-\log t
-
\frac{4}{\log n}
-
\frac{4}{t\log n}
\le
\widehat\mu_n^1[(t,1]].
\end{align*}
This proves \eqref{eq:reverse-shifted-tail}.\\

\noindent
By Lemma~\ref{Propcouplmonotone}, we may list the atoms of
\(\widehat{\mathcal X}_n^1\) and \(\mathcal X_n^2\) as
\((\widehat X_k^1)_{k\ge1}\) and \((X_k^2)_{k\ge1}\), respectively, in such a way that
\begin{align*}
\widehat X_k^1\le X_k^2
\end{align*}
for every \(k\ge1\). Moreover, the reverse shifted tail estimate \eqref{eq:reverse-shifted-tail} implies that these matched atoms satisfy
\begin{align*}
X_k^2-\widehat X_k^1\le 2\delta_n,
\end{align*}
for $k\ge1$. Therefore, setting
\begin{align*}
N_n:=\widehat{\mathcal X}_n^1(\mathbbm X),
\end{align*}
we have
\begin{align*}
R_n
=
S_n^2-\widehat S_n^1
=
\sum_{k\ge1}(X_k^2-\widehat X_k^1)
\le
2\delta_n N_n.
\end{align*}
Since \(\widehat{\mathcal X}_n^1\) has control measure \(\widehat\mu_n^1\), the random variable \(N_n\) is Poisson with mean
\begin{align*}
\widehat\mu_n^1((0,1])
=
\mu_n^2((0,1])
\le
C\log\log n.
\end{align*}

\noindent By Chernoff's bound,
\begin{align*}
\Pb[R_n>r_n]
\le
C(\log n)^{-2},
\end{align*}
provided \(K_1\) is chosen large enough. Similarly, since \(Z_n\) has bounded mean and \(x_n\) is of the order $1/\log n$,
\begin{align*}
\Pb[x_nZ_n>r_n]
\le
C(\log n)^{-2}.
\end{align*}

\noindent
It remains to control the concentration terms. Define
\begin{align*}
\psi_n(r)
:=
\Pb[1-r\le S_n^1\le1],
\end{align*}
for $0\leq r\leq 1/2$. Recall that
\begin{align*}
\mathcal X_n^{\mathrm{geo}}
:=
\sum_{p\in\mathcal P_n}\xi_p\,
\delta_{\frac{\log p}{\log n}},
\end{align*}
and set
\begin{align*}
S_n^{\mathrm{geo}}
:=
\int_{\mathbbm X}x\,\mathcal X_n^{\mathrm{geo}}(dx).
\end{align*}
By construction,
\begin{align*}
S_n^{\mathrm{geo}}-S_n^1
=
\frac{1}{\log n}
\sum_{p\in\mathcal P_n}
\sum_{k\ge2}k\,N[\{(p,k)\}]\,\log p.
\end{align*}
It follows that, for \(0\leq r\leq 1/2\) and \(\alpha>0\),
\begin{align*}
\{1-r\le S_n^1\le1\}
&\subseteq
\{1-r\le S_n^{\mathrm{geo}}\le1+\alpha\}
\\
&\qquad\cup
\left\{
\frac{1}{\log n}
\sum_{p\in\mathcal P_n}
\sum_{k\ge2}k\,N[\{(p,k)\}]\,\log p>\alpha
\right\}.
\end{align*}
By Lemma \ref{lem:kolmmain} and the boundedness of \(\rho\) on compact intervals, there exists \(C>0\) such that
\begin{align*}
\Pb[1-r\le S_n^{\mathrm{geo}}\le1+\alpha]
\le
\frac{C}{\log n}
+
C(r+\alpha).
\end{align*}
Moreover, fixing \(t\in(0,1/2)\), exponential Chebyshev gives
\begin{align*}
\Pb\left[
\frac{1}{\log n}
\sum_{p\in\mathcal P_n}
\sum_{k\ge2}k\,N[\{(p,k)\}]\,\log p>\alpha
\right]
\le
C_t n^{-t\alpha},
\end{align*}
where
\begin{align*}
C_t
:=
\exp\left\{
\sum_{p\in\mathcal P}
\sum_{k\ge2}
\frac{1}{kp^{k(1-t)}}
\right\}
<\infty.
\end{align*}
Consequently,
\begin{align*}
\psi_n(r)
\le
\frac{C}{\log n}
+
C(r+\alpha)
+
C_t n^{-t\alpha}.
\end{align*}

\noindent
Choose
\begin{align*}
\alpha_n
:=
\frac{V(t\log n)}{t\log n},
\end{align*}
where \(V\) denotes the inverse function of \(u\mapsto ue^u\) on \([0,\infty)\). Then
\begin{align*}
n^{-t\alpha_n}
=
\alpha_n
\le
C\frac{\log\log n}{\log n}.
\end{align*}
Therefore
\begin{align*}
\psi_n(r_n)
\le
C\frac{\log\log n}{\log n}.
\end{align*}
Since
\begin{align*}
\widehat S_n^1=S_n^1+x_nZ_n,
\end{align*}
we have
\begin{align*}
\Pb[1-r_n\le \widehat S_n^1\le1]
&\le
\Pb[x_nZ_n>r_n]
+
\Pb[1-2r_n\le S_n^1\le1]
\le
\Pb[x_nZ_n>r_n]
+
\psi_n(2r_n).
\end{align*}
Combining the preceding estimates gives
\begin{align*}
\mathcal E_{A^o}[\mathcal X_n^2,\mathcal X_n^1]
\le
C\frac{\log\log n}{\log n}.
\end{align*}
Since \(\mathrm e-\frac32>1\), this also gives
\begin{align*}
\mathcal E_{A^o}[\mathcal X_n^2,\mathcal X_n^1]
\le
C\frac{(\log\log n)^{\mathrm e-\frac32}}{\log n}.
\end{align*}

\subsection{Approximation by the Dickman control measure}

The goal of this section is to construct an \(\mathbf N(\mathbbm X)\)-valued Poisson random variable \(\mathcal X_n^3\) with control measure
\begin{align*}
\mu_n^3(dx)
:=
\frac{1}{x}\mathbf 1_{(0,1]}(x)\,dx,
\end{align*}
and to show that \(\mathcal X_n^3\) is close to \(\mathcal X_n^2\) both in the Wasserstein distance associated with \(\mathfrak d\) and in the error term \(\mathcal E_{A^o}\).

Recall that
\begin{align*}
x_n:=\frac{\log 2}{\log n},
\qquad
\delta_n:=\frac{K_0}{\log n},
\qquad
y_n:=x_n+\delta_n,
\end{align*}
and that
\begin{align*}
\mu_n^2
=
(\Phi_n)_{\#}\nu_n
+
a\delta_{y_n},
\qquad
\Phi_n(x):=(x+\delta_n)\wedge1,
\end{align*}
where
\begin{align*}
\nu_n(dx)
=
\frac{1}{x}\mathbf 1_{(x_n,1]}(x)\,dx.
\end{align*}

\noindent Let \(\mathcal N_n^{(0)}\) and \(\mathcal N_n^{(1)}\) be independent Poisson point processes with control measures
\begin{align*}
\frac{1}{x}\mathbf 1_{(0,x_n]}(x)\,dx,
\qquad
\frac{1}{x}\mathbf 1_{(x_n,1]}(x)\,dx,
\end{align*}
respectively. Let \(Z_n\) be an independent Poisson random variable with parameter \(a\). Define
\begin{align*}
\mathcal X_n^3
&:=
\mathcal N_n^{(0)}+\mathcal N_n^{(1)},
\\
\mathcal X_n^2
&:=
(\Phi_n)_{\#}\mathcal N_n^{(1)}
+
Z_n\delta_{y_n}.
\end{align*}
By construction, \(\mathcal X_n^3\) has control measure \(\mu_n^3\), while \(\mathcal X_n^2\) has control measure \(\mu_n^2\). We first estimate \(\E[\mathfrak d(\mathcal X_n^3,\mathcal X_n^2)]\). By the definition of \(\mathfrak d\) and the above coupling,
\begin{align*}
\mathfrak d(\mathcal X_n^3,\mathcal X_n^2)
&\le
\int_{(0,x_n]}x\,\mathcal N_n^{(0)}(dx)
+
\int_{(x_n,1]}|\Phi_n(x)-x|\,\mathcal N_n^{(1)}(dx)
+
y_nZ_n.
\end{align*}
Since \(|\Phi_n(x)-x|\le\delta_n\), taking expectations gives
\begin{align*}
\E[\mathfrak d(\mathcal X_n^3,\mathcal X_n^2)]
&\le
\int_0^{x_n}dx
+
\delta_n\int_{x_n}^1\frac{dx}{x}
+
ay_n
=
x_n
+
\delta_n\log\left(\frac1{x_n}\right)
+
ay_n
\le
C\frac{\log\log n}{\log n}.
\end{align*}

\noindent Since \(\mathrm e-\frac32>1\), this also implies
\begin{align}\label{eq:Wd_X3_X2}
\E[\mathfrak d(\mathcal X_n^3,\mathcal X_n^2)]
\le
C\frac{(\log\log n)^{\mathrm e-\frac32}}{\log n}.
\end{align}

\noindent We now estimate \(\mathcal E_{A^o}[\mathcal X_n^3,\mathcal X_n^2]\), with \(A^o\) as in \eqref{eq:Aodef}.
Define
\begin{align*}
S_n^{(i)}
:=
\int_{\mathbbm X}x\,\mathcal X_n^{i}(dx),
\end{align*}
for $i=2,3$. Under the above coupling,
\begin{align*}
|S_n^{(3)}-S_n^{(2)}|
\le
R_n,
\end{align*}
where
\begin{align*}
R_n
:=
\int_{(0,x_n]}x\,\mathcal N_n^{(0)}(dx)
+
\int_{(x_n,1]}|\Phi_n(x)-x|\,\mathcal N_n^{(1)}(dx)
+
y_nZ_n.
\end{align*}
Thus, for every \(\varepsilon>0\),
\begin{align*}
\mathcal E_{A^o}[\mathcal X_n^3,\mathcal X_n^2]
&\le
\Pb\big[
|S_n^{(3)}-1|\le |S_n^{(3)}-S_n^{(2)}|
\big]
\\
&\le
\Pb\big[
|S_n^{(3)}-1|\le \varepsilon
\big]
+
\Pb[R_n>\varepsilon].
\end{align*}

\noindent Since \(S_n^{(3)}\) has the Dickman distribution, with density \(e^{-\gamma}\rho\), and \(\rho\) is bounded on compact intervals, there exists \(C>0\) such that
\begin{align*}
\Pb\big[
|S_n^{(3)}-1|\le \varepsilon
\big]
\le
C\varepsilon.
\end{align*}

\noindent We now bound the second term. Take
\begin{align*}
\varepsilon_n:=K_1\frac{\log\log n}{\log n},
\end{align*}
with \(K_1>0\) large enough. We decompose \(R_n\) according to its three terms. First, since the jumps of
\begin{align*}
\int_{(0,x_n]}x\,\mathcal N_n^{(0)}(dx)
\end{align*}
are bounded by \(x_n\), an exponential Chebyshev bound gives
\begin{align*}
\Pb\left[
\int_{(0,x_n]}x\,\mathcal N_n^{(0)}(dx)>\frac{\varepsilon_n}{3}
\right]
\le
C(\log n)^{-2},
\end{align*}
provided \(K_1\) is large enough.\\

\noindent Second,
\begin{align*}
\int_{(x_n,1]}|\Phi_n(x)-x|\,\mathcal N_n^{(1)}(dx)
\le
\delta_n\,\mathcal N_n^{(1)}((x_n,1]).
\end{align*}
The random variable \(\mathcal N_n^{(1)}((x_n,1])\) is Poisson with mean
\begin{align*}
\nu_n((x_n,1])
=
\log\left(\frac1{x_n}\right)
\le
C\log\log n.
\end{align*}
Therefore, by Chernoff's bound,
\begin{align*}
\Pb\left[
\int_{(x_n,1]}|\Phi_n(x)-x|\,\mathcal N_n^{(1)}(dx)>\frac{\varepsilon_n}{3}
\right]
\le
C(\log n)^{-2},
\end{align*}
again by choosing \(K_1\) large enough.\\

\noindent Finally, since \(Z_n\) has fixed mean \(a\) and \(y_n=O(1/\log n)\), another Chernoff bound gives
\begin{align*}
\Pb\left[
y_nZ_n>\frac{\varepsilon_n}{3}
\right]
\le
C(\log n)^{-2}.
\end{align*}
Combining the preceding estimates,
\begin{align*}
\Pb[R_n>\varepsilon_n]
\le
C(\log n)^{-2}.
\end{align*}
Consequently,
\begin{align*}
\mathcal E_{A^o}[\mathcal X_n^3,\mathcal X_n^2]
&\le
C\varepsilon_n
+
C(\log n)^{-2}
\\
&\le
C\frac{\log\log n}{\log n}.
\end{align*}
Since \(\mathrm e-\frac32>1\), this also gives
\begin{align*}
\mathcal E_{A^o}[\mathcal X_n^3,\mathcal X_n^2]
\le
C\frac{(\log\log n)^{\mathrm e-\frac32}}{\log n}.
\end{align*}

\subsection{Conditioning step}\label{sec:conditioningstep}

\noindent We now transfer the estimates obtained for the unconditioned approximation chain to the corresponding conditioned laws. We first record a simple conditioning estimate.

\begin{lemma}\label{lem:conditioning-transfer}

Let \(X,Y,Z\) be \(\mathbf N(\mathbbm X)\)-valued random variables defined on a common probability space, and let \(B\subset \mathbf N(\mathbbm X)\) be measurable. Suppose that

\begin{align*}
p
:=
\Pb[Z\in B]\wedge \Pb[Y\in B]
>
0.
\end{align*}

Assume moreover that \(X\in B\) on the event \(\{Z\in B\}\), and that \(B\) has finite diameter with respect to \(\mathfrak d\), namely

\begin{align*}
D_B
:=
\sup_{\eta,\zeta\in B}\mathfrak d(\eta,\zeta)
<
\infty.
\end{align*}
Then
\begin{align*}
d_W\left(
\mathcal L(X\mid Z\in B),
\mathcal L(Y\mid Y\in B)
\right)
\le
\frac{\E[\mathfrak d(X,Y)]}{p}
+
\frac{2D_B}{p}\,
\Pb\big[\{Z\in B\}\Delta\{Y\in B\}\big].
\end{align*}
\end{lemma}

\begin{proof}
Let
\begin{align*}
B_Z:=\{Z\in B\},
\qquad
B_Y:=\{Y\in B\},
\end{align*}
and write
\begin{align*}
p_Z:=\Pb[B_Z],
\qquad
p_Y:=\Pb[B_Y],
\qquad
p:=p_Z\wedge p_Y.
\end{align*}
We first use the original coupling on the event \(B_Z\cap B_Y\). On this event, the transportation cost is \(\mathfrak d(X,Y)\). After conditioning, this contribution is bounded by
\begin{align*}
\frac{1}{p}\,
\E\big[
\mathfrak d(X,Y)\mathbbm 1_{B_Z\cap B_Y}
\big].
\end{align*}

\noindent It remains to account for the part of the two conditional laws that is not paired through the event \(B_Z\cap B_Y\). Besides the masses coming from
\(B_Z\setminus B_Y\) and \(B_Y\setminus B_Z\), there may also be a mismatch of normalization on the common event \(B_Z\cap B_Y\), since the factors \(p_Z^{-1}\) and \(p_Y^{-1}\) need not agree. The total unmatched mass is bounded by
\begin{align*}
&\frac{\Pb[B_Z\setminus B_Y]}{p_Z}
+
\frac{\Pb[B_Y\setminus B_Z]}{p_Y}
+
\Pb[B_Z\cap B_Y]\left|\frac1{p_Z}-\frac1{p_Y}\right|
\\
&\le
\frac{2\Pb[B_Z\Delta B_Y]}{p}.
\end{align*}

\noindent On this remaining part, both variables belong to \(B\): indeed, \(X\in B\) on \(B_Z\) by assumption, and \(Y\in B\) on \(B_Y\) by definition. Therefore the transportation cost on the unmatched part is at most \(D_B\). We obtain
\begin{align*}
d_W\left(
\mathcal L(X\mid B_Z),
\mathcal L(Y\mid B_Y)
\right)
&\le
\frac{1}{p}\,
\E\big[
\mathfrak d(X,Y)\mathbbm 1_{B_Z\cap B_Y}
\big]
+
\frac{2D_B}{p}\Pb[B_Z\Delta B_Y]
\\
&\le
\frac{\E[\mathfrak d(X,Y)]}{p}
+
\frac{2D_B}{p}\Pb[B_Z\Delta B_Y].
\end{align*}
The proof is complete.
\end{proof}

\noindent We apply Lemma~\ref{lem:conditioning-transfer} with
\begin{align*}
X=\mathcal X_n^0,
\qquad
Z=\mathcal X_n^{\mathrm{geo}},
\qquad
Y=\mathcal X_n^3,
\qquad
B=A^o.
\end{align*}
Since
\begin{align*}
\int_{\mathbbm X}x\,\mathcal X_n^0(dx)
\le
\int_{\mathbbm X}x\,\mathcal X_n^{\mathrm{geo}}(dx),
\end{align*}
we have \(\mathcal X_n^0\) belongs to $A^o$ on the event \(\{\mathcal X_n^{\mathrm{geo}}\in A^o\}\). Then, by \eqref{eq:trueconditionedlaw},
\begin{align*}
\mathcal L\left(X\mid Z\in A^o\right)
=
\mathcal{L}\!\left(T^n_{\#}\big[\iota\circ \mathfrak{P}(H_n)\big]\right).
\end{align*}
Moreover, by Proposition~\ref{p:div_H},
\begin{align*}
\Pb[Z\in A^o]
=
\Pb[A_n]
\ge
\frac12.
\end{align*}
On the other hand, by the Dickman density of the total mass of the limiting Poisson point process,
\begin{align*}
\Pb[Y\in A^o]
=
\Pb\left[
\int_{\mathbbm X}x\,Y(dx)\le1
\right]
=
e^{-\gamma_{\mathrm E}},
\end{align*}
because \(\rho(t)=1\) for \(0\le t\le1\). Therefore
\begin{align*}
\Pb[Z\in A^o]\wedge \Pb[Y\in A^o]
\ge
c_0
\end{align*}
for some universal constant \(c_0>0\). Also, if \(\eta,\zeta\in A^o\), then
\begin{align*}
\mathfrak d(\eta,\zeta)
\le
\int_{\mathbbm X}x\,\eta(dx)
+
\int_{\mathbbm X}x\,\zeta(dx)
\le
2,
\end{align*}
so \(D_{A^o}\le2\).\\

\noindent Using Lemma~\ref{lem:conditioning-transfer}, together with \eqref{eq:total-event-discrepancy} and \eqref{eq:total-metric-discrepancy}, we obtain
\begin{align*}
d_W\left(
\mathcal{L}\!\left(T^n_{\#}\big[\iota\circ \mathfrak{P}(H_n)\big]\right),
\mathcal L(\mathcal X_n^3\mid \mathcal X_n^3\in A^o)
\right)
\le
C\frac{(\log\log n)^{\mathrm e-\frac32}}{\log n}.
\end{align*}
Finally, by Lemma~\ref{lem:attenuated-PD-poisson-cloud},
\begin{align*}
\mathcal L\big(\mathfrak q(\mathcal X_n^3)\mid \mathcal X_n^3\in A^o\big)
=
\mathfrak A[\theta_{\mathcal S}].
\end{align*}
Passing from point measures to their decreasing atom sequences, and using that this map is an isometry between \(\mathfrak d\) and \(d_{\mathcal S}\) on the configurations under consideration, we conclude that
\begin{align*}
d_{W,\mathcal S}
\big(
\vartheta_n,
\mathfrak A[\theta_{\mathcal S}]
\big)
\le
C\frac{(\log\log n)^{\mathrm e-\frac32}}{\log n}.
\end{align*}
The proof is now complete.

\appendix

\section{Proof of Lemma  \ref{Propcouplmonotone}}\label{sec}

\noindent Define the tail functions
\[
\Lambda_i(t):=\mu_i[(t,1]],
\]
for $i=1,2$ and their generalized inverses
\[
Q_i(s):=\sup\{t\in(0,1]\ ;\ \Lambda_i(t)\ge s\},
\]
for $s>0$, with the convention that the supremum of the empty set is zero. Let \((\Gamma_k\ ;\ k\ge1)\) be the points of a Poisson process with rate \(1\) on \((0,\infty)\), listed in increasing order, and define
\[
X_k^i:=Q_i(\Gamma_k),
\]
for $k\ge1$ and $ i=1,2.$ By the standard inverse-tail construction of Poisson point processes,
\[
\sum_{k\ge1}\mathbbm 1_{\{X_k^i>0\}}\delta_{X_k^i}
\]
has law \(\rho_{\mu_i}\), for \(i=1,2\). Since \(\mu_1\preceq\mu_2\), we have
\[
\Lambda_1(t)\le \Lambda_2(t),
\]
for all $0<t\le1.$ Therefore
\[
\{t\in(0,1]\ ;\ \Lambda_1(t)\ge s\}
\subseteq
\{t\in(0,1]\ ;\ \Lambda_2(t)\ge s\},
\]
for every \(s>0\). Hence \(Q_1(s)\le Q_2(s)\) for every \(s>0\), and consequently
\[
X_k^1\le X_k^2,
\]
for $k\ge1.$ The result follow from here.\\

\noindent \textbf{Acknowledgements}\\
Arturo Jaramillo Gil was supported by
the grant CBF2023-2024-2088.

\bibliographystyle{plain}
\bibliography{Bib}

\end{document}